\documentclass[11pt,a4paper,reqno,final]{article}
\newif\iffinal
\finaltrue

\usepackage{graphicx}
\usepackage{bbm}
\usepackage{template}
\usetikzlibrary{patterns}
\usepackage{sidecap}
\usepackage{fullpage}
\usepackage{subcaption}

\DeclareMathOperator\Geo{Geo}

\def\scalezero{L}
\def\expdefect{4}
\def\expfractal{{10}}

\begin{document}

\title{A new proof for percolation phase transition on stretched lattices}

\date{\today}
\author{
  Marcelo Hil\'ario \thanks{UFMG -  Departement of Mathematics, Av.\ Antonio Carlos 6627, 31270-901 Belo Horizonte, MG - Brazil.}
  \and
  Marcos S\'a \thanks{
    IMPA - Estrada Dona Castorina 110, 22460-320 Rio de Janeiro, RJ - Brazil and UFBA - Instituto de Matem\'atica, Av.\ Milton Santos, 40170-110  Salvador, BA - Brazil.}
  \and
  Remy Sanchis \footnotemark[1]
  \and
  Augusto Teixeira \thanks{
    IMPA - Estrada Dona Castorina 110, 22460-320 Rio de Janeiro, RJ - Brazil and IST - University of Lisbon, Portugal.}}
\maketitle

\begin{abstract}
  In this article, we revisit the phase transition for percolation on randomly stretched lattices.
  Starting with the usual square grid, keep all vertices untouched while erasing edges according to the following procedure: for every integer $i$, the entire column of vertical edges contained in the line $\{ x = i \}$ is removed independently of other columns with probability $\rho > 0$.
  Similarly, for every integer $j$, the entire row of horizontal edges contained in the line $\{ y = j\}$ is removed independently of other rows with probability $\rho$.
  On the remaining random lattice, we perform Bernoulli bond percolation.
  Our main contribution is an alternative proof that the model undergoes a nontrivial phase transition, a result which was earlier established by Hoffman.
  The main novelty of our work lies on the fact that the dynamic renormalization employed earlier is now replaced by a static version, which is easier to master and more robust to extend to different models.
  We emphasize the flexibility of our methods by showing the non-triviality of the phase transition for a new oriented percolation model in a random environment as well as for a model previously investigated by Kesten, Sidoravicius and Vares.
  In addition, we prove a result about the sensitivity of the phase transition with respect to the stretching mechanism and provide a list of open problems that could be explored using our techniques.

  \vspace{3mm}


\end{abstract}

\section{Introduction}

The Bernoulli bond percolation model on the $d$-dimensional lattice has drawn great attention since it appeared in the mathematics literature \cite{PSP:2048852} as one of the simplest models that features a non-trivial phase transition.
Despite its simple construction, the model remains a source of fascinating open problems and longstanding conjectures, see for instance \cite{Gri99,bollobas2006percolation} for a comprehensive exposition to the subject.
Besides the interest on $\mathbb{Z}^d$, many works have focused on the study of the phase diagram for percolation on top of various random graphs, both as a way to model inhomogeneities in the environment \cite{van2010percolation,bramson1991,beaton2021alignment} but also to understand the geometry of the underlying random graph \cite{Hoffman2005,haggstrom2001coloring}.
For such models, it is often the case that even showing the existence of a non-trivial phase transition becomes a challenging task.

In this article we revisit a family of percolation models in random environment first introduced in \cite{jonasson2000percolation} and further studied in \cite{Hoffman2005}.
The original process is defined as follows.
Starting with the usual nearest neighbor square lattice, independently for every $i \in \mathbb{Z}$ and with probability $\rho \in [0,1)$, remove all vertical edges lying along the line $\{i\} \times \mathbb{R}$.
Repeat the procedure for the perpendicular direction, deleting independently with probability $\rho$ all horizontal edges that lie at a line $\mathbb{R} \times \{j\}$, for $j \in \mathbb{Z}$.
One is now left with a randomly stretched lattice $\mathcal{G}$, which is a dilute subgraph of the $\mathbb{Z}^2$-lattice with random separations between consecutive rows and columns (see Figure~\ref{f:lattice} for an illustration).
More details on this construction will be provided in Section \ref{s:notation}.

\begin{figure}[h]
  \begin{subfigure}{.55\textwidth}
    \centering
    \includegraphics{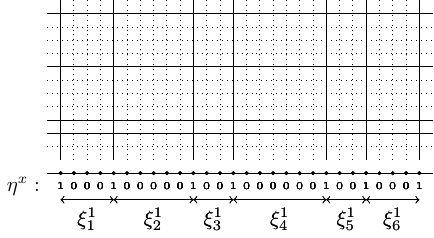}
    \caption{}
    \label{f:lattice}
  \end{subfigure}
  \begin{subfigure}{.44\textwidth}
    \centering
    \includegraphics{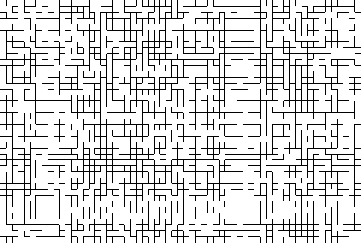}
    \caption{}
    \label{f:simulation}
  \end{subfigure}
  \caption{(a) A possible outcome for a portion of the random graph $\mathcal{G}$.
    The sequence $\eta^x$ of zeros and ones determine whether we delete or keep the vertical columns.
    (b) A small-scale simulation of the percolation process performed on top of our random graph.
    Infinite-range dependencies along the horizontal and vertical direction become evident.
  }
\end{figure}

Conditioned on the realization of $\mathcal{G}$, we now perform Bernoulli bond percolation on this graph with parameter $p \in [0,1]$. That is, every edge in $\mathcal{G}$ is kept (respectively, removed) with probability $p$ (respectively, $1-p$)  independently of the others, see Figure~\ref{f:simulation}.

Since $\mathcal{G}$ is a subgraph of $\mathbb{Z}^2$,  percolation does not occur if $p \leq 1/2$, see \cite{harris1960lower}.
It was conjectured in \cite{jonasson2000percolation} (see Conjecture~3.2) that for some $\rho > 0$ and some $p < 1$ there is almost surely an infinite connected percolation component in $\mathcal{G}$.
This result has later been proved by Hoffman in \cite{Hoffman2005}.
Our first contribution is to provide the following quantitative version of Hoffman's result:
\begin{theorem}
  \label{t:main}
 Let $\scalezero=10^6$.
 If $\rho \leq \scalezero^{-16}$ and $p \geq (1 - \scalezero^{-10})^{1/4}$, then
  \begin{equation}
    \label{e:percolation}
    \mathbb{P} \big[ (0, 0) \text{ is connected to infinity} \big] > 0.
  \end{equation}
\end{theorem}

\begin{remark}
  \begin{enumerate}[\quad a)]
  \item The constant $L = 10^6$ appearing in the statement of Theorem \ref{t:main} should be regarded as a suitable scaling factor, and apart from the fact that it needs to be chosen sufficiently large, its exact value will not play an important role.
    Moreover, we have not put much effort in optimizing its value (and others appearing in the text), favoring instead the readability of the proofs.
  \item Theorem \ref{t:main} provides explicit ranges for the values of the parameters $\rho$ and $p$ for which percolation occurs.
    By means of a standard one-step block argument, one can obtain from Theorem~\ref{t:main} that for every $\rho > 0$ there exists $p<1$ sufficiently large so that \eqref{e:percolation} holds.
    Although we do not reproduce such argument here, the reader is directed to Section~3 of \cite{bramson1991} where a similar procedure is implemented.
    Also, one can show that for every $p > p_c(\mathbb{Z}^2) = 1/2$ there exists $\rho > 0$ sufficiently small so that \eqref{e:percolation} holds.
    The one-step renormalization argument that leads to this improvement is similar to the proof we present in Subsection~\ref{ss:oriented_gives_ksv}.
  \end{enumerate}
\end{remark}


At this point we would like to discuss the novelty of our contribution.
Renormalization is an important technique for the study of statistical mechanics problems, due to its robustness with respect to the microscopic details of the model, which allows it to be extended to various similar processes.

However, many models in statistical mechanics involve two sources of randomness, such as the diluted Ising model or our current analysis of percolation on a random graph.
In such situations, one often needs to employ two separate renormalizations procedures: one for the random environment and another for the ensuing process.

When using two separate renormalization arguments, it is often the case that the first one is made dynamic \cite{bramson1991, Hoffman2005, Kesten22, gacs2011clairvoyant, basu2014lipschitz}.
By this we mean that the sizes of the boxes, as well as their positions are made random and determined only after an algorithmic inspection of the environment.
This leads to a model-dependent proof, since each environment has different microscopic features, imposing difficulties in extending the reasoning to other models.
Moreover, these algorithmic constructions are more technically involved.

To give an example supporting the claim that static renormalizations are easier to extend, consider the article \cite{bramson1991}, where the authors use a dynamic renormalization to study the contact process with random defects (this model is analogous to the one studied here with the lattice stretched along a single direction).
In \cite{HSS19}, an alternative static proof was established, yielding a much stronger result with respect to the stretching mechanism.

Such improvement can also be noticed in the current article, where we also replace the dynamic renormalization introduced in \cite{Hoffman2005} with another argument employing a static renormalization instead.
Note that Theorem~\ref{t:main} cannot be made stronger in the sense introduced in \cite{HSS19}, as emphasized by Theorem~\ref{t:sharpness} below.
However, the static renormalization that we have developed is much more flexible and we make this clear by providing a considerably shorter proof of the phase transition for another model studied in \cite{Kesten22}.
This extension is given in Section~\ref{s:ksv}.

\bigskip

Our next result addresses the question of how stable the phase transition is with respect to changes in the stretching mechanism.
Before stating it, we need to change our perspective on how the model is constructed.

Observe that by removing the columns and rows independently with probability $\rho$, one effectively obtains independent $\Geo(\rho)$ gaps between present (non-removed) columns and rows, see Figure~\ref{f:lattice}.
In fact, one can alternatively define the model as follows: let $\xi^x=\{\xi^x_i\}_{i\in \mathbb{Z}}$ and $\xi^y=\{\xi^y_j\}_{j \in \mathbb{Z}}$ be i.i.d. sequences representing the horizontal and vertical gaps, respectively.
Then, each edge in the standard, nearest neighbor square lattice is open or closed independently with the following probabilities.
A horizontal edge $\{ (i, j), (i + 1, j) \}$ is declared open with probability $p^{\xi^x_i}$ (respectively a vertical edge $\{ (i, j), (i, j + 1) \}$ with probability $p^{\xi^y_j}$).
For a fixed realization of $\xi^x$ and $\xi^y$, and for a parameter $p \in (0,1)$, let $\mathbb{P}^{\xi^x,\xi^y}_p$ denote the law governing the resulting percolation process in $\mathbb{Z}^2$.
In this new formulation, the choice for geometric separation seems rather arbitrary, however the following theorem shows that it is indeed sharp.

\begin{theorem}
Let $\xi^x=\{\xi^x_i\}_{i \in \mathbb{Z}}$ and $\xi^y=\{\xi^y_j\}_{j\in\mathbb{Z}}$ be independent sequences of i.i.d.\ random variables such that:
  \label{t:sharpness}
  \begin{enumerate}[a)]
  \item \label{l:no_transition:1} $\xi^x_0$ does not have compact support, and
  \item \label{l:no_transition:2} the tails of $\xi^y_0$ are heavier than geometric, i.e.
    \begin{equation*}
      \text{for every $\mu \in (0, 1)$, } \quad \limsup_{t\to\infty} \mu^{-t} \mathbb{P}(\xi^y_0 > t) = \infty.
    \end{equation*}
    or equivalently, for all $t > 0$, $E(e^{t\xi^y_0}) = \infty$.
  \end{enumerate}
  Then for almost every environment $\xi^x$ and $\xi^y$,
\begin{equation}
  \text{$\mathbb{P}_p^{\xi^x,\xi^y} \big[ (0, 0) \text{ is connected to infinity} \big]=0$, every $p \in (0,1)$.}
\end{equation}
\end{theorem}

\begin{remark}
  \label{r:compact_support}
  If the marginals of $\xi^x$ have compact support, then the phase transition persists under much weaker hypotheses on the tails of the marginals of $\xi^y$.
  This was the central result in \cite{HSS19}, where it is proved that, if $E((\xi^y_0)^{1 + \epsilon}) < \infty$ for some $\epsilon >0$, then $\mathbb{P}_p^{\xi^x, \xi^y}(o\leftrightarrow\infty) > 0$ for every $p < 1$ sufficiently large.
  In \cite{HSS19} it is also shown that $E(\xi_0^y) < \infty$ is a necessary condition for percolation.
\end{remark}

\begin{remark}
In \cite{HSS19} it has been proved that the radius of an open cluster below the critical point follows a power-law decay.
In \cite{jahnel2022continuum} the authors also show a power-law decay for a continuum percolation model similar to the one considered here.
\end{remark}

\begin{remark}
The existence of a non-trivial phase transition for percolation on the stretched lattice, immediately suggests that one can investigate the same type of phenomenon for other classes of models originating from statistical mechanics.
 As a first step towards that direction, let us mention the work by Häggström \cite{haggstrom_2000} which allows us to directly translate results on the phase transition for a class of models.
In particular, given that the graphs studied here have bounded degree, one can immediately deduce that stretched lattices feature a non-trivial phase transition for: Bernoulli site percolation, the Ising model, the Widom-Rowlinson model and the Beach model.
\end{remark}

\paragraph{Related works.}
Since the pioneering work of McCoy, Wu \cite{mccoy68} on the Ising Model, many works have turned to the question of how the phase transition of a statistical mechanics model is affected by the presence of an environment having columnar disorder.
In the last decades several works have considered that question in the context of percolation,
see for instance \cite{duminil2018brochette, Kesten22, bramson1991, jahnel2022continuum, delima2022dependent}.

A series of works have dealt with a planar percolation process where the state of a given vertex $(i, j)$ is specified by random variables $\xi^x_i$ and $\xi^y_j$ living on the coordinate axes.
Such percolation models are often referred to as coordinate percolation, see for example
\cite{10.1007/978-981-15-0302-3_4, basu2014lipschitz, gacs2011clairvoyant, moseman2008form, winkler2000dependent, pete2008, marchand2022corner} for some important progresses in these models.
It is also worth mentioning some models that have been studied in higher dimensions, such as Bernoulli line percolation \cite{hilario2019linepercolation} and others carrying infinite range dependence, see \cite{Szn09,TW10b}.

\paragraph{Overview of the proofs.}

The proofs of Theorems~\ref{t:main} and \ref{t:oriented} follow a similar structure which puts in evidence the suitability of static renormalization to approach situations where dependencies are present.
Both proofs are divided in two steps that we call the `Control of the Environment' and the `Control of Percolation', and that are respectively detailed in Sections~\ref{s:environment} and \ref{s:percolation}.

In the analysis of the environment, we employ a multi-scale renormalization on $\mathbb{Z}$, where the sequence of length scales grows exponentially as $\scalezero^k$,  where $L$ should be regarded as a sufficiently large number (we fix $L = 10^6$ for the sake of concreteness).

At each scale, we pave $\mathbb{Z}$ using disjoint intervals with the corresponding length.
For each of these intervals $I$ we introduce a notion of ``difficulty'' that, roughly speaking, quantifies how hard one would find to cross a column having $I$ as a basis in the resulting percolation process.
This quantifier is a suitable function of the environment $\xi$ in $I$ and assumes values in $\{0, 1, 2, \dots\}$.
The main goal of Section \ref{s:environment} is to prove that typical intervals have difficulty equal to zero provided $\rho$ is sufficiently small.
One can then assign difficulties to intervals both in the $x$-axis as well as the $y$-axis, using the sequences $\xi^x$ and $\xi^y$.

We start Section~\ref{s:percolation} already knowing that a typical box at scale $k$ projects to good intervals (with zero difficulty) both in the $x$-axis and in the $y$-axis.
Our main objective is then to prove that such a box contains a large percolation cluster with high probability provided $p$ is close enough to one.
This is done by induction and involves a recursive estimate on the probability to cross columns with varying ``difficulties''.

The central idea of the proof lies in the assignment of ``difficulty levels'' to different intervals.
Certain choices of this ``assignment rule'' would make the renormalization step for the environment easy, while the percolation step unfeasible.
Other choices could act the opposite way, making it so that typical intervals are bad.
We strike a balance in our choice of ``difficulty rule'' that provides a tradeoff between the two renormalization steps, in the following sense:
\begin{itemize}
\item Difficulty labels have a tendency to diminish from one scale to another.
  This makes the first renormalization argument work, proving that a typical large interval is good (difficulty~$0$).
\item At the same time, the difficulty assigned to a large interval takes good account of the difficulty labels in sub-intervals.
  This allows the second renormalization to work, by proving that: once a box projects to good intervals in both axes, it typically contains a large connected component.
\end{itemize}
We later explain how this balance is attained.

\paragraph{Acknowledgements}

MH has been supported by CNPq grants ``Projeto Universal'' (406001/2021-9) and ``Produtividade em Pesquisa'' (312227/2020-5), and by FAPEMIG grant `Projeto Universal' (APQ-01214-21).
MS has received funds from CNPq grant ``Programa de Capacita\c c\~ao Institucional'' and by Funarbe (5812).
RS has been partially supported by CNPq (312837/2020-8), CAPES, and FAPEMIG (APQ-00868-21, RED-00133-21).
During this period, AT has been supported by grants ``Projeto Universal'' (406250/2016-2) and ``Produtividade em Pesquisa'' (304437/2018-2) from CNPq and ``Jovem Cientista do Nosso Estado'', (202.716/2018) from FAPERJ.

\paragraph{Organization of the paper}
In Section~\ref{s:notation} we introduce some more notation that is required for our proofs.
Section~\ref{s:environment} contains the necessary estimates on the environment, proving that a typical box projects to good intervals both horizontally and vertically.
In Sections~\ref{s:percolation}, \ref{s:contraction} and \ref{s:summability} we complete the proof of Theorem~\ref{t:main} by analyzing the percolation process when the environment is controlled.
Theorem~\ref{t:sharpness} is proved in Section~\ref{s:sharpness}, while the oriented models are treated in Section~\ref{s:ksv}.

\section{Definition of the model and further notation.}
\label{s:notation}

We now present the construction of the model in a precise mathematical way and introduce notation that will be used throughout the paper.
As usual, the two-dimensional Euclidean lattice is $(\mathbb{Z}^2, \mathcal{E})$, where $\mathcal{E} = \big\{ \{z, w\}; |z - w| = 1 \big\}$.
It will be convenient to regard that lattice naturally embedded in $\mathbb{R}^2$.
Denote $e_1, e_2$ the canonical basis for $\mathbb{R}^2$ and $\pi_x, \pi_y$ the projections onto the horizontal and vertical coordinates, respectively.
In what follows it will be useful to formulate the model in two equivalent ways.

First we consider a pair of independent two-sided sequences $\eta^x=(\eta^x_i)_{i \in \mathbb{Z}}$, $\eta^y=(\eta^y_i)_{i \in \mathbb{Z}}$ of  i.i.d.\ $\Ber(\rho)$ random variables, where $\rho \in [0, 1]$.
Given $\eta^x$ and $\eta^y$ we define a new graph $\mathcal{G} = \big( \mathbb{Z}^2, \mathcal{E}(\eta^x, \eta^y) \big)$ called a stretched lattice, whose edge set is $\mathcal{E}(\eta^x, \eta^y) = \mathcal{E}^h(\eta^y) \cup \mathcal{E}^v(\eta^x)$, where
\begin{equation}
  \label{e:graph}
  \begin{split}
    \mathcal{E}^h(\eta^y) & = \big\{ \{z, z + e_1\};\, \eta^y(\pi_y(z)) = 1 \big\},\\
    \mathcal{E}^v(\eta^x)  & =\big\{ \{z, z + e_2\};\, \eta^x(\pi_x(z)) = 1 \big\}
  \end{split}
\end{equation}
Then we define on the appropriate measurable space the probability measure $P^{\eta^x, \eta^y}_p$, under which the configurations $\omega \in \{0, 1\}^{\mathcal{E}(\eta^x, \eta^y)}$ are distributed as an i.i.d.\ bond percolation on $\big( \mathbb{Z}^2, \mathcal{E}(\eta^x, \eta^y) \big)$ with parameter $p$.

It is easy to map the percolation process on the stretched lattice $\mathcal{G}$ back to the usual $\mathbb{Z}^2$ lattice.
For that end, let  $\xi^x = (\xi^x_i)_{i \in \mathbb{Z}}$ (respectively $\xi^y = (\xi^y_j)_{j \in \mathbb{Z}}$) stand for the amount of $0$'s between two consecutive occurrences of $1$'s in $\eta^x$ (respectively $\eta^y$).
With this notation, $\xi^x$ and $\xi^y$ are independent two-sided sequences of i.i.d.\ random variables with distribution $\Geo(1 - \rho)$:
\begin{equation}
  \label{e:distr_xi}
  \mathbb{P}(\xi^x_i \geq n) = \rho^n, \qquad \mathbb{P}(\xi^y_j \geq n) = \rho^n, \qquad n \in \mathbb{Z}_+
\end{equation}
Note that these random variables are allowed to assume value $0$.

Conditional on $\xi^x$, $\xi^y$, and for $p \in [0, 1]$ we define a bond percolation on the usual $\mathbb{Z}^2$ by declaring  each horizontal edge $\{(i, j), (i + 1, j)\}$ open with probability $p^{\xi^x_i + 1}$ and closed otherwise while each vertical edge $\{(i, j), (i, j + 1)\}$ is declared open with $p^{\xi^y_j + 1}$.
The states of the edges are mutually independent.
We denote $P^{\xi^x, \xi^y}_p$ the law of the process on the appropriate probability space.

One can easily verify that the formulations in terms of $\eta$'s and $\xi$'s are equivalent.

\begin{remark}
  Standard ergodicity arguments may be employed to define $p_c(\rho)$ as being the critical value for $p$ above which an infinite cluster exists almost surely (under the annealed law, i.e.\ the law of the percolation process averaged over the realization of the environment).
  Theorem \ref{t:main} provides ranges of parameters $(\rho,p)$ for which $p > p_c(\rho)$, so that an infinite cluster exists almost surely (under the annealed law).
  Following the standard argument of Burton and Keane \cite{BK89} for the annealed law, one may prove uniqueness of that infinite cluster for all the range $p > p_c(\rho)$.
  In fact the resulting percolation is translation invariant and ergodic with respect to lattice shifts in the direction spanned by the vector $e_1+e_2=(1,1)$.
  Moreover, in the formulation of the model given in \eqref{e:distr_xi}, it becomes clear that the percolation enjoys the so-called finite-energy property.
\end{remark}

\section{Environment}
\label{s:environment}

This section is devoted to the control of the environment through a renormalization scheme  designed to analyze the behavior of a sequence $\xi$ of geometric random variables (see \eqref{e:distr_xi}) over large length scales.

We start by defining the sequence of scales.
First, fix $\scalezero = 10^6$ and let
\begin{equation}
  \label{e:def_Lk}
  L_k = \scalezero^k,  \textrm{ for $k\in \mathbb{Z}_+$.}
\end{equation}
For each $i\in \mathbb{Z}$, we define the $i$-th $k$-interval (or $i$-th interval of scale $k$) as
\begin{equation}
  \label{e:def_Ik}
  I_{(k, i)} = \big[ i L_k, (i + 1) L_k \big) \cap \mathbb{Z}.
\end{equation}

One can immediately verify that, for every $k \in \mathbb{Z}_+$:
  \begin{enumerate}[i)]
  \item $0\in I_{(k,0)}$;
  \item the intervals $\{I_{(k, i)}\}_{i \in \mathbb{Z}}$ form a disjoint paving of $\mathbb{Z}$;
  \item \label{e:ttI_union} an $(k+1)$-interval 
   is composed of exactly $\scalezero$ sub-intervals of scale $k$.
  \end{enumerate}
It is worth to mention that the resulting paving of $\mathbb{Z}$ is `static' meaning that it does not depend on the realization of the environment $\xi$.
Moreover, the choice of $\scalezero=10^6$ is made for convenience, and any sufficiently large value would work equally well
(here, due to several choices made throughout the text, it will suffice that $L\geq 2^{15}+1$, see \eqref{e:L_is_very_large}).

Let us introduce for each $k \geq 0$ the set $M_k := \{k\} \times \mathbb{Z}$ whose elements $m \in M_k$ will serve to index intervals $I_m$ and also random variables $H_m$ (to be defined soon) that will quantify how bad the environment inside $I_m$ is.
Once these random variables are defined we will be able to label each of the intervals $I_m$ either good or bad according to the following criterion:
\begin{display}
\label{e:def_good_int}
For $m \in M_k$, $I_m$ is said \emph{good} if $H_m = 0$
and \emph{bad} otherwise.
\end{display}
We will often abuse notation and write `\emph{$m$ is good (bad)}' instead of `\emph{$I_m$ is good (bad)}' and, indeed we may refer to $m$ and $I_m$ indistinguishably.

The definition of $H_m$ will be recursive, and made in such a way that we expect most of the intervals to be good.
But, for those that are bad, the values $H_m > 0$ should be understood as a measure of the intensity of the defect inside $I_m$, large values of $H_m$ corresponding to large intensity of defects.

\medskip

\paragraph{Defining $H_m$ for $m \in M_0$.}
For each $m = (0, i) \in M_0$, define
\begin{equation}
  \label{e:param_0}
  H_{(0, i)} = \xi_i.
\end{equation}
In view of \eqref{e:def_good_int}, an interval $I_{(0, i)}$ is good if and only if $\xi_i = 0$.

\begin{remark}
  \label{r:exponential_decay_h}
  Although the choice in \eqref{e:param_0} is rather arbitrary, it is reasonable to require the intensity of a defect to be given by the corresponding variable $\xi_i$ since the latter is a simple measurement of how hard it will be for the resulting percolation to traverse the column $\{i\} \times \mathbb{Z}$.
  \end{remark}

  \begin{remark}
  It is straightforward to verify from the choice of $H_m$ in \eqref{e:param_0} and from \eqref{e:distr_xi} that, for every $m = (0, i) \in M_0$,
  \begin{equation}\label{e:dist_geo_scalezero}
    \mathbb{P}[H_m = h] = \mathbb{P}[\xi_i = h] = \rho^{h} (1 - \rho) \leq \rho^h.
  \end{equation}
  This may be taken as an indication on how to define the intensity for defects at higher scales.
  Given the i.i.d.\ nature of our random environment, if we want to preserve the exponential decay at higher scales, the probability to find defects with intensity $h_1, \dots, h_r$ in separate intervals should decay exponentially with $h_1 + \dots + h_r$.
  Therefore, it is natural that the defect intensity at scale $k + 1$ be related to the sum of the defects of its sub-intervals from scale $k$.
\end{remark}

Inspired by this argument, let us now move to higher scales using recursion.
For that we write for $m =  (k+1, i) \in M_{k+1}$
\begin{equation}
  \label{e:Q_m}
  \mathcal{Q}_m =\{m' \in M_{k} \colon I_{m'} \subset I_m\} =\{(k, i \scalezero  + j) \colon 0 \leq j \leq L-1\}.
\end{equation}

\paragraph{Defining $H_{m}$, for $m \in M_{k + 1}$.}
Assuming that we already know the values of $H_{m'}$, for every $m' \in M_k$ we recall the definition of $\mathcal{Q}_m$ in \eqref{e:Q_m} and define
\begin{equation}
  \label{e:H_m}
  H_m =
  \begin{cases}
    0, \qquad & \text{if all intervals $m' \in \mathcal{Q}_m$ are good},\\
    H_{m_1} -1, & \text{if $m_1$ is the only bad sub-interval in $\mathcal{Q}_m$},\\
    1 + \displaystyle\sum_{i = 1}^{r} H_{m_i}, & \text{if $\{m_1, \dots, m_r\}$, with $r\geq2$, are the bad sub-intervals in $\mathcal{Q}_m$}.
  \end{cases}
\end{equation}

  \begin{remark}
    There is yet another important reason why the choice for the defect at scale $k+1$ should be related to the sum of the defects at scale $k$.
    As mentioned above, it is convenient to think of the value of $H_m$ as a measurement of the intensity (or depth) of the defect exhibited by the environment inside the interval $I_m$.
    In fact, when studying the percolation process on $\mathcal{G}$, we will deal with the probability of crossing vertical strips that projects to the defective intervals $I_{m_i}$.
    It will turn out that, starting from a large set of points $S$ to the left-hand side of such a strip, we will be able to reach a set $R$ of points at the right-hand side, whose cardinality behaves roughly as $|R| = |S| e^{- \alpha(p) H_{m_i}}$.
    Therefore, as we cross successive strips with defects the multiplicative impact on the cardinality loss amounts to an additive term in the exponent.
    This hints that we should add up the intensities $H_{m_i}$ of the defects at successive strips to get the intensity of a larger strip containing them.
    This reasoning is formalized at Lemma \ref{l:algebra}.
    \label{r:exponential_percolation}
  \end{remark}

\begin{remark}
Despite the reasoning in the preceding remark, in the above definition we do not simply add the defects of the previous scale.
  In fact, we reward the occurrence of a single defect (by subtracting one unit to $H_{m_1}$) while penalizing the occurrence of several (by summing one unit).
  The following items may clarify the reason why that is done.
  \begin{enumerate}[i)]
  \item In case a single bad sub-interval exists, we subtract one unit to account for the recovery provided by the good surroundings of that sub-interval.
In fact, from the perspective of the higher scale, the existence of a single sub-interval provides plenty of room for the percolation process to build clusters next to each side of the resulting defective strip and hence several opportunities to connect these clusters across that strip (this will be quantified in the Lemma \ref{l:recovery}).
  \item The defect of an interval that contains two or more bad sub-intervals from the previous scale is strictly larger than the sum of the forming defects due to the addition of one unit in the third term in \eqref{e:H_m}.
  This helps to simplify future calculations, see for instance \eqref{e:importance_2}.
  \item One of our main goals is to prove that the probability of observing deep defects decays fast.
    This would be easier to establish if larger defects would always arise as combinations of two (or more) defects from bad subintervals.
    In fact, the occurrence of two already independent unlikely events in the previous scale should be very improbable.
    However, we also need to consider the event that a defect is generated from a single bad subinterval, whose probability is harder to bound because we are left with a single unlikely event whose cost gets amortized by the entropy that comes from the amount of possible positions of that bad subinterval.
    This hints why we can afford to add one unit to the last term, and why it is convenient to subtract one unit in the middle term in \eqref{e:H_m}.
  \end{enumerate}
\end{remark}

\bigskip

Another feature that we will make use of throughout the text is the following:
\begin{display}
  \label{e:good_single_bad}
  a good interval $m$ may contain at most one bad sub-interval, in which case that bad sub-interval $m'\in Q_m$ necessarily has $H_{m'} = 1$.
\end{display}
It is straightforward to deduce it from \eqref{e:H_m}.

\bigskip

\paragraph{Controlling the probabilities of bad intervals.} Our main goal is to show that the environment is well-behaved, meaning that we will most likely observe only good intervals over large scales.
In order to quantify this notion, we define the following sequence
\begin{equation}
  \label{e:p_k}
  p_k
  = \sup_{m \in M_k} \mathbb{P} [\text{$m$ is bad}]
  = \mathbb{P} [\text{$(k, 0)$ is bad}].
\end{equation}
Here, the second equality holds thanks to translation invariance.
Our goal boils down to show that $p_k$ decays fast as we increase $k$ which we are able to do for sufficiently small values of $\rho$.

\begin{lemma}[Environment]
  \label{l:environment}
  For every $\rho \leq \scalezero^{-16}$, we have
  \begin{equation}
    \label{e:environment}
    p_k \leq \scalezero^{-k/2}=L_k^{-1/2}, \text{ for every $k \geq 0$.}
  \end{equation}
\end{lemma}

Instead of proving Lemma \ref{l:environment} directly, we will show that it is a consequence of a stronger version (Lemma~\ref{l:decay_khb}) that takes into account more refined features of the defects besides their intensities.
In order to state that version we will need to introduce yet another family of variables $B_m$.
Intuitively speaking, they will register the latest scale at which two or more separate defects were joined together in order to form the defect observed inside $I_m$.

The definition is recursive.
At scale zero, we simply define $B_m = 0$ for each $m = (0, i)$, $i  \in \mathbb{Z}$.
Now, suppose that for every $m' \in M_k$, we have already defined the variables $B_{m'}$.
Then, for any $m \in M_{k+1}$, we define
\begin{equation}
  \label{e:B_m}
  B_m =
  \begin{cases}
    0, \qquad & \text{if $m$ is a good interval,}\\
    B_{m'}, & \text{if $m$ is bad and $m'$ is the only bad sub-interval  in $\mathcal{Q}_m$},\\
    k + 1, & \text{if there is more then one bad sub-intervals in $\mathcal{Q}_m$}.
  \end{cases}
\end{equation}
The reader may find it useful to consult Figure \ref{f:bm} for an illustration.

\begin{figure}[htb]
  \centering
  \includegraphics{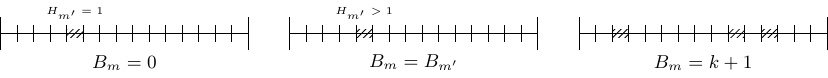}
  \caption{Possible values of $B_m$ for an $m \in M_{k+1}$. On the left, there is a single bad sub-interval with $H=1$ so $m$ is good ($H_m=0$) and $B_m = 0$.
    On the middle there is again a single bad sub-interval, but this time with $H >1$ ($H_m>0$) so $B_m$ takes the $B$-value of this sub-interval.
    On the right, there are at least two bad sub-intervals, implying that $B_m$ equals the scale index $k+1$.}
  \label{f:bm}
\end{figure}

\begin{remark}
  \label{r:B_m}
  Below we list some of the properties of the variables $B_m$:
  \begin{enumerate}[i)]
  \item If $I_m$ is bad, then $B_m$ represents the largest scale at which two or more defects have been combined to form the defect observed in $I_m$.
  \item The choice $B_m=0$ for good intervals is arbitrary, and it will not be used later.
  \item It may be the case that $m$ is still declared good despite containing a single bad sub-interval $m'$.
    In that case, \eqref{e:B_m} implies $B_m=0$.
    It is only when $m'$ is bad enough for the larger interval $m$ to be declared bad  we assign $B_m = B_{m'}$.
  \end{enumerate}
\end{remark}

The salient properties of the $B_m$ $(m \in M_k$, $k \geq 1$) that we use throughout the text are,
\begin{display}
  \label{e:B_m_large}
  if $B_m = k$, then the interval $I_m$ is bad \\ and it contains
  at least two bad sub-intervals,
\end{display}
and
\begin{display}
  \label{e:B_m_small}
  if $ B_m < k$, then the interval $I_m$ \\ contains at most one bad sub-interval.
\end{display}
It is important to observe that some proofs throughout the paper will distinguish between the cases $B_m = k$ and $B_m < k$ as above.

In Lemma \ref{l:decay_khb} instead of $p_k$ we will bound the quantities
\begin{equation}
  \label{e:def_pkhb}
  p_{k,h,b} = \sup_{m \in M_k} \mathbb{P} [ H_m = h, B_m = b ] = \mathbb{P} [ H_{(k, 0)} = h, B_{(k, 0)} = b ]
\end{equation}
obtaining a much finer control of the environment since we are taking into account not only how bad the interval is ($H_m = h$), but also the structure of the defect therein ($B_m = b$):

\begin{lemma}
  \label{l:decay_khb}
  For any $\rho < \scalezero^{-16}$, we have
  \begin{equation}
    \label{e:environment2}
    p_{k, h, b} \leq \scalezero^{- \big( k + \frac{h}{b+1} \big) - 2 h - 13},
  \end{equation}
  for every $k, b \geq 0$ and $h \geq 1$.
\end{lemma}

Before we prove the above lemma, let us quickly show how it leads to Lemma~\ref{l:environment}.

\begin{proof}[Proof of Lemma~\ref{l:environment}]
  The proof follows from the simple calculation:
  \begin{equation}
    \begin{split}
      \mathbb{P}[(k,0) \text{ is bad}]
      & \overset{\eqref{e:def_good_int}}{=}  \sum_{h = 1}^{\infty} \mathbb{P}[H_{(k,0)} = h]
        = \sum_{h = 1}^{\infty} \sum_{b=0}^{k} \mathbb{P}[H_{(k,0)} = h, B_{(k,0)}=b] \\
      & \overset{\eqref{e:def_pkhb}}{\leq} \sum_{h = 1}^{\infty} \sum_{b=0}^{k} p_{k,h,b}
        \overset{\eqref{e:environment2}}{\leq} \sum_{h = 1}^{\infty} \sum_{b=0}^{k} \scalezero^{- \big( k + \frac{h}{b+1} \big) - 2 h - 13} \\
      & \leq (k+1) \sum_{h = 1}^{\infty} \scalezero^{- k - 2 h - 13}\\
      &  = (k+1) \frac{\scalezero^{-k-15}}{1- \scalezero^{-2}} \leq \scalezero^{-k/2},
    \end{split}
  \end{equation}
where the last inequality is trivially satisfied for all $k\geq 0$.
\end{proof}

We finish this section by providing its missing piece.

\begin{proof}[Proof of Lemma~\ref{l:decay_khb}]
  We start with the case $k = 0$ where the intervals are single points.
  Recall that $\rho < \scalezero^{-16}$ and that, since $k = 0$, we have $B_{(0, 0)} = 0$.
  Thus, for every $h \geq 1$,
  \begin{equation*}
    p_{0, h, 0} = \mathbb{P} [ H_{(0, 0)} = h ] \overset{\eqref{e:dist_geo_scalezero}} \leq \rho^h \leq \scalezero^{-3h - 13} = \scalezero^{- \big( 0 + \frac{h}{0 + 1} \big) - 2 h - 13},
  \end{equation*}
where the second inequality sign follows from our choice of $\rho$ and the fact that t$h\geq 1$.
  This establishes \eqref{e:environment2} for $k = 0$.

  Suppose now that for some $k \geq 0$, \eqref{e:environment2} holds and let us complete the induction by proving that it also holds for $k + 1$.
  For this, consider the interval $m = (k + 1, 0)$, some $h \geq 1$ and let us split the argument into two cases depending on the value of $b$.
  \medskip

  {\bf Case 1 - ($b<k+1$):}

   In this case, \eqref{e:H_m} and \eqref{e:B_m} imply that there is a single bad sub-interval $m'$ in $\mathcal{Q}_{m}$ with $H_{m'}=h+1$ and $B_{m'} = B_m = b$.
  Since $\mathcal{Q}_{m}=\{k\}\times\{0, 1, \dots, \scalezero-1 \}$,
  \begin{equation}
    \label{e:p_k+1,h,b_1}
    \begin{split}
      p_{k+1,h,b} & = \mathbb{P} \big[ H_m = h, B_m = b \big]
                    \leq  \mathbb{P} \big[ \mcup_{i = 0}^{\scalezero-1} [H_{(k, i)} = h + 1, B_{(k, i)} = b ] \big]
                    \leq \scalezero \,\, p_{k,h+1,b}
    \end{split}
  \end{equation}
  Now using that we have assumed \eqref{e:environment2} to hold for $k$ we have
  \begin{equation}
    \label{e:p_k+1,h,b_2}
    \scalezero \,\, p_{k,h+1,b} \leq \scalezero^{ 1 - \big(k + \frac{ h + 1 }{ b + 1 }\big)- 2( h+1 ) -13} \leq \scalezero^{-\big( (k + 1) + \frac{h}{ b+1 }\big) - 2 h - 13}.
  \end{equation}
  Plugging \eqref{e:p_k+1,h,b_2} into \eqref{e:p_k+1,h,b_1} yields the desired bound \eqref{e:environment2} for $k + 1$.

  \medskip

  {\bf Case 2 - ($b = k + 1$):}

  Since $B_{m} = b = k + 1$,
  by \eqref{e:B_m}, there are $r\geq 2$ bad sub-intervals $m_i$,  $i \in \{1, \dots, r \}$, with $\{m_1, \dots, m_r\} \subseteq \mathcal{Q}_{m}$.
  Using \eqref{e:H_m} and writing $H_{m_i} = h_i \geq 1$ we have
  \begin{equation}
    \label{e:choices_h}
    \sum_{i = 1}^r h_i = h - 1.
  \end{equation}
  We also write $B_{m_i} = b_i \in [0, k]$ to decompose
  \begin{equation}\label{e:decomposicao_dos_pks}
    p_{k + 1, h, b} = \sum_{r = 2}^{\scalezero}\;\;
    \sum_{\substack{\{m_1, \dots, m_r\} \\ \text{in } \mathcal{Q}_m}} \;\;
    \sum_{\substack{h_1, \dots, h_r\geq 1; \\ \sum_i h_i = h - 1}} \;\;
    \sum_{0\leq b_1, \dots, b_r\leq k} \;\;
    \mathbb{P} \Big[ \mcap_{i = 1}^r \big[ H_{m_i} = h_i, B_{m_i} = b_i \big] \Big].
  \end{equation}
  Recalling our assumption that \eqref{e:environment2} holds for $k$, we obtain
  \begin{equation}\label{e:multiplicacao_dos_pks}
    \begin{split}
      \mathbb{P} \Big[ & \mcap_{i = 1}^r \big[ H_{m_i} = h_i, B_{m_i} = b_i \big] \Big]\\
                       & = \prod_{i=1}^r p_{k,h_i,b_i} \leq \prod_{i=1}^r \scalezero^{-\big(k+\frac{h_i}{b_i+1}\big)-2h_i - 13}
                         \overset{\eqref{e:choices_h}}\leq \scalezero^{-\big(rk+\frac{h-1}{k+1}\big)-2 (h-1) - 13r}.
    \end{split}
  \end{equation}
  Since there are at most $\scalezero^{r}$ choices for the intervals $m_1, \ldots, m_r$, $h^r$ choices for $h_1, \ldots, h_r$ and $(k+1)^r$ choices for $b_1, \ldots, b_r$ we have by \eqref{e:decomposicao_dos_pks} and \eqref{e:multiplicacao_dos_pks} that
  \begin{equation}
    \begin{split}
      p_{k + 1, h, k + 1}
      & \leq \sum_{r = 2}^{\scalezero} \scalezero^{ - \big( rk + \frac{h - 1}{k + 1} \big) - 2(h - 1) - 13r} \cdot \scalezero^{r} \cdot h^r \cdot (k + 1)^r\\
      & \leq \sum_{r = 2}^{\scalezero} \scalezero^{-  \big( rk + \frac{h}{k + 1} \big) + 2 - 2h + 2 - 13r + r + r \log_{\scalezero} h + r \log_{\scalezero} (k + 1)}\\
      & = \sum_{r = 2}^{\scalezero} \scalezero^{- rk - \frac{h}{k + 1} - 2h + 4 - 12r + r \log_{\scalezero} h + r \log_{\scalezero}(k + 1)}.
    \end{split}
  \end{equation}
  Dividing both sides by our desired estimate on $p_{k + 1, h, k + 1}$ (see \eqref{e:environment2}), we obtain
  \begin{equation}
    \label{p_k+1overM}
    \dfrac{p_{k + 1, h, k + 1}}{\scalezero^{-\big( (k + 1) + \frac{h}{(k + 1) + 1} \big) - 2h - 13}}
    = \dfrac{p_{k + 1, h, k + 1}}{\scalezero^{- k - \frac{h}{k + 2} - 2h - 14}}
    \leq \sum_{r= 2}^{\scalezero}\scalezero^{\varphi(r,h,k)},
  \end{equation}
  where
  \begin{equation}
    \varphi(r,h,k)
    := - (r - 1) k - \frac{h}{(k + 1)(k + 2)} - 12 r + 18 + r \log_{\scalezero} (h) + r \log_{\scalezero}(k + 1),
  \end{equation} is defined for $k\geq 0$, $h\geq 1$ and $2\leq r\leq \scalezero$.

  We will be done once we show that the right-hand side in \eqref{p_k+1overM} is bounded above by $1$.
  For that, it suffices to show that $\varphi(r,h,k) < -r$ for any $r\in[2,\scalezero]$, because in this case the right-hand side in \eqref{p_k+1overM} will then be bounded by $\sum_{r \geq 2} \scalezero^{-r} < 1$ as desired.
  We will prove the bound $\varphi < -r$ by considering two cases depending on whether $h$ is larger than $h_* := \scalezero [k \; (k + 1) (k + 2) + 1]$ or not.

  We first consider the case $h \geq h_*$.
  Then as $h \geq k + 1$,
  \begin{equation}
    r \log_{\scalezero} (h) + r \log_{\scalezero}(k + 1)
    \leq 2  r  \log_{\scalezero} (h)
  \end{equation}
  and we now claim that

  \begin{equation}
    \label{e:bound_1200logh}
    2r \log_{\scalezero}(h) \leq \frac{h}{(k + 1)(k + 2)} + 2r, \text{ for every $h \geq h_*$}.
  \end{equation}
  Indeed, one can first verify this bound on $h = h_*$, which can be done for every $k \geq 0$ using that $h_*\leq \scalezero(k+1)^3$ and $6r \log_{\scalezero}(k+1)\leq 6\cdot \scalezero \log_{\scalezero} (k+1)\leq \scalezero k$.
  One can then show that the derivative of the left hand side is bounded by the derivative on the right for every $h \geq h_*$ and $k \geq 0$.
  Having \eqref{e:bound_1200logh} we can use the fact that $9r \geq 18$ to conclude that $\varphi < -r$ as desired.

  We now deal with the case $h \leq h_* < \scalezero(k + 1)^3$.
  In this case, for any $r \geq 2$ and $k \geq 0$,
  \begin{equation}
    r \log_{\scalezero} (h) + r \log_{\scalezero}(k + 1)
    < r \log_{\scalezero} \big( \scalezero (k + 1)^4 \big)
    = 4r\log_{\scalezero}(k + 1) + r \leq (r - 1)k + r,
  \end{equation}
  where the last inequality follows from $r/2 \leq r - 1$ for $r\geq 2$ and $8 \log_{\scalezero}(k + 1)\leq k$ for $k\geq 0$.
  Using once again that $9r \geq 18$, we obtain $\varphi < -2r < -r$.
  This finishes the proof of the lemma.
\end{proof}

\section{Traversal of good and bad columns}
\label{s:percolation}

The renormalization scheme presented in Section \ref{s:environment} was aimed at controlling the defects of the one-dimensional environment encoded by the sequence $\xi$.
An important ingredient was the multiscale defect classification for the collection of intervals $I_m$ used to pave the one-dimensional space $\mathbb{Z}$ .
From now on, we consider randomness in both directions (horizontal and vertical) encoded respectively by the sequences $\xi^x$ e $\xi^y$ as in \eqref{e:distr_xi}, and we turn our attention to the bond percolation process on $\mathbb{Z}^2$ distributed according to $\mathbb{P}_p^{\xi^x, \xi^y}$.

It will be convenient to consider horizontal and vertical intervals separately, denoted by $I^x_m$ and $I^y_m$, respectively.
Similarly, for $d \in \{x, y\}$, we define the associated random variables $H^d_m$ (c.f.\ \eqref{e:H_m}) and $B^d_m$ (c.f.\ \eqref{e:B_m}).
Each of the intervals $I^x_m$ and $I^y_m$ will be labeled good or bad depending respectively on the random variables $H_m^x$ or $H_m^y$ (c.f.\ \eqref{e:def_good_int}).

Let us now give an intuitive description of the content of the present section.
We will suppose that $\xi^x$ and $\xi^y$ are such that the intervals $I^x_{(k,0)}$ and $I^y_{(k,0)}$ are good for every $k\geq 0$, which has positive probability by Lemma \ref{l:environment}.
Given any such $\xi^x$ and $\xi^y$, we wish to show that, for $p$ sufficiently large, there exists an infinite percolation cluster $\mathbb{P}^{\xi^x, \xi^y}_p$ a.s.
In order to do so, we will show that rectangles of type $I^x_{(k+1,0)} \times I^y_{(k,0)}$ are very likely crossed in the hard direction.
This allows one to build an infinite cluster by gluing crossings in a standard fashion.

Even though $I^x_{(k+1,0)}$ is assumed to be good, it may contain one bad subinterval within (see \eqref{e:good_single_bad}).
Hence in order to cross such a rectangle, we will be forced to traverse a bad rectangle of the type $I^x_{(k,j)} \times I^y_{(k,0)}$.
Bearing in mind the recursive definition of $H^x_{(k,j)}$, as one inspects the microscopic structure of the defect inside $I^x_{(k,j)}$, one may find out that it is a combination of defects with higher labels which are present at lower scale subintervals.
Therefore, we will need to actually study crossings of rectangles of type $I^x_{(k',j)} \times I^y_{(k,0)}$ with $k' < k$ and with arbitrary high values of the defect $H^x_{(k',j)}$.

The purpose of the next sections is to build a multiscale scheme that allows us to obtain a detailed quantitative analysis of these crossing events depending on the intensity of the underlying defects.

\color{black}

\subsection{Fractal sets and remainder}

As mentioned above, in order to find a percolation path, we will eventually be forced to cross very difficult vertical strips (columns) and horizontal strips (rows).
An example of such a difficult column would be $I^x_m \times \mathbb{Z}$, where a large label $H^x_m$ has been assigned to the interval $I^x_m$.
Therefore, the environment can be regarded as a collection of traps extending along the coordinate directions and labeled according to their difficulty.

To cross such a trap, say a bad vertical strip, the strategy will be to fix a large set $S \subseteq \mathbb{Z}$ of possible starting points at one side of the bad strip and understand which of these starting points will extend to a crossings all the way to the opposite side of the strip.
We now introduce some geometric definitions that will help us characterize such sets $S$.

For every $S \subseteq \mathbb{Z}$ and every scale $k$ we denote
\begin{equation}
  \label{e:Z_k_A}
  Z_k(S) := \big\{ m \in M_k; I_m \cap S \neq \varnothing \big\}
\end{equation}
the collection of indices $m$'s at scale $k$, whose associated intervals $I_m$ 
intersect $S$.
For any pair of non-empty collections $\color{black}Z, Z' \subseteq M_k$
\begin{display}
  \label{e:monotone_M}
  we say that $Z \prec Z'$ if
  $(k, i) \in Z$ and $(k, j) \in Z'$ imply $i < j$.
\end{display}
In particular, if $Z \prec Z'$ then $Z$ and $Z'$ are disjoint.
We can now introduce the concepts of $k$-ordered sets which will be useful later.
\begin{definition}[$k$-ordered sets]
  Let $\xi$ be given. A family $S_1, \dots, S_\ell$ of subsets of $\mathbb{Z}$ is said to be $k$-ordered if $Z_k(S_1) \prec \dots \prec Z_k(S_\ell)$.
\end{definition}

\begin{definition}[$k$-good sets]
  \label{d:good_set}
  A set $S$ is said to be $k$-good if all intervals in $Z_k(S)$ are good according to $\xi^y$.
\end{definition}

We are now in position to define $k$-fractal sets, introduced recursively as follows:
\begin{definition}[$k$-fractal sets]
  \label{d:fractal}
  Let $\xi$ be given.
  A singleton $S=\{i\}$ in $\mathbb{Z}$ is said $0$-fractal if $i$ is good, that is if $\xi_i = 0$.
  For $k \geq 1$, a subset $S$ of $\mathbb{Z}$ is a $k$-fractal if $S = S_1 \cup \dots \cup S_{2^\expfractal}$, where
\begin{itemize}
\item $S$ is $k$-good;
\item the family $\{S_0, \dots, S_{2^\expfractal}\}$ is $(k-1)$-ordered and
\item each $S_i$ is a $(k - 1)$-fractal.
 \end{itemize}
\end{definition}
We observe that every $k$-fractal set $S$ has cardinality $|S|=2^{\expfractal k}$.

\begin{remark}
The reason why we chose to write powers of two such as $2^\expfractal$ instead of their more current form $1024$ is  to emphasize the effect that crossing defects will have on the exponents.
\end{remark}

\begin{figure}[h]
  \centering
  \includegraphics[width=0.8\textwidth]{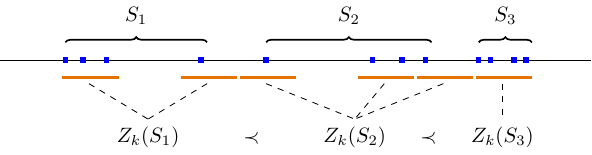}
  \caption{A $k$-ordered family of sets $S_1, S_2$ and $S_3$.
    The orange $k$-intervals are good.
    They are disjoint and form the collections $Z_k(S_1)$, $Z_k(S_2)$ and $Z_k(S_3)$.}
\end{figure}

The requirement that $k$-fractal sets be unions over $(k-1)$-ordered families imposes a sparseness property because they must be contained in disjoint collections of $(k-1)$-intervals.
This sparseness will be important later, since it will allow us to explore independence when we try to find crossings in parallel disjoint strips.
However, in some situations it will be important to disallow $k$-fractal sets to be too sparse and for such cases we introduce the following property
\begin{equation}
  \label{e:grouped}
   \text{a $k$-fractal $S \subseteq \mathbb{Z}$ is said to be \emph{$k$-grouped} if $S \subseteq I_m$ for some $m \in M_k$}.
\end{equation}

As an example on how to generate a $k$-fractal set which is also $k$-grouped, one can start with a good $k$-interval $I_m$, $m\in M_k$, and choose $2^\expfractal$ indices in $\mathcal{Q}_m$ corresponding to good $(k-1)$-intervals.
For each of these $(k-1)$-intervals, choose again $2^\expfractal$ of its $(k-2)$-subintervals that are good.
Iterate until finding $2^{\expfractal k}$ good points (i.e.\ good $0$-intervals), which will form a $k$-grouped, $k$-fractal set.

We finish this section by providing some extra definitions that will be important in order to establish our exploration procedure of the configuration inside different rectangles.

We define a \emph{rectangle} $R = [a, b) \times [c, d)$ as being the subgraph of $\mathbb{Z}^2$ whose sites belong to $[a,b] \times [c,d]$ and with edge set $\mathcal{E} = \big\{ \{z, w\}; z \in [a,b] \times [c,d], w \in [a,b) \times [c,d), |z - w| = 1 \big\}$.
Note that we slightly abuse notation, since we add to $R = [a, b) \times [c, d)$ the sites on its \emph{right} and \emph{top faces}, $F^r(R)=\{b\} \times [c, d]$ and $F^t(R)=[a,b] \times \{d\}$, including also edges that connect to $F^r$ and $F^t$.
That is done in order to guarantee that adjacent rectangles, sharing a face, have disjoint set of edges which will help in the matter of exploring independence, while we do not worry about the sites in their intersection.
We also write $F^l(R)=\{a\} \times [c,d]$ and $F^b(R)=[a,b] \times \{c\}$ respectively for the \emph{left face} and \emph{bottom face} of $R$.

Note that for all $k\geq 0$, the rectangles
\begin{equation}
  \label{e:boxes}
  R_{(k,i,j)} := I_{(k,i)} \times I_{(k,j)}\overset{\eqref{e:def_Ik}} = \big[iL_k, (i+1)L_k\big) \times[jL_k + , (j + 1) L_k)\textrm{ for } i,j\in\mathbb{Z}^2
\end{equation}
are edge-disjoint and pave $\mathbb{Z}^2$ (although they share vertices in their common faces).
Observe also that $R_{(0,i,j)}$ only contains two edges, $(i,j)\sim (i + 1,j)$ and $(i,j)\sim (i, j + 1)$.

Given a rectangle $R = [a, b) \times [c, d)$ and a realization of the percolation in the edges of $R$, we associate for each set $S \subseteq \mathbb{Z}$ such that $S \times \{a\} \subseteq F^l(R)$, the \emph{remainder} of $S$ defined as the set of points in the right face of $R$ that can be connected from $S \times \{a\}$ within $R$, or more precisely:
\begin{equation}
  \label{e:remainder_box}
  \mathcal{R}(S, R) = \Big\{ z \in F^r(R); \text{ some $w \in S\times \{a\}$ is connected to $z$ within $R$} \Big\}.
\end{equation}

\begin{remark}
In what follows, when the rectangle $R$ and the set $S$ are given, we will find it convenient not to distinguish between $S$ and $S \times \{a\}$.
Therefore, we sometimes regard a set $S$ as being a subset of $F^l(R)$.
\end{remark}

In order to recover independence, it will be useful to force the above crossings to take place inside  disjoint horizontal slices.
To make that precise, we define: for any given scale $k \geq 0$, any finite interval $I \subseteq \mathbb{Z}$ and any set $S \subseteq \mathbb{Z}$
\begin{equation}
  \label{e:remainder_box_A}
  \mathcal{R}_k(S, I) = \bigcup_{m' \in Z_k(S)} \mathcal{R}\big(S \cap  I^y_{m'}, I \times I^y_{m'}\big),
\end{equation}
see Figure~\ref{f:remainder}.
Note that paths starting from the vertices in $S \cap I^y_{m'}$ are required to remain within a single box $R=I \times I^y_{m'}$, a requirement that is also illustrated in Figure~\ref{f:remainder}.
Also, if we have $S\subset I^y_m$ for some $m\in M_k$ (that is $S$ is $k$-grouped), then $\mathcal{R}_k(S,I) = \mathcal{R}(S, I\times I^y_m)$.

\begin{figure}[h]
\centering
 \includegraphics[width=.5\textwidth]{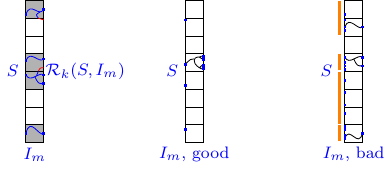}
\caption{On the left: An illustration of $S$ and its remainder $\mathcal{R}_k(S,I_m)$.
Crossings have to take place inside their corresponding box (notice that the red connections are invalid).
On the middle: When $I_n$ is good, even if $S$ is sparse, the remainder is expected to contains a $k$-grouped $k$-fractal set.
On the right: When $I_m$ is bad, if $S$ is a sufficiently large family the remainder is expected to contain a $k$-fractal (although one does not expect it to be $k$-grouped).
\label{f:remainder}
}
\end{figure}

\subsection{Grouping and traversing}

Starting from a set of points $S$ on the left face of a rectangle, we explore the percolation configuration towards the right inside that rectangle.
There are important aspects of to be taken into consideration during this exploration: whether the column is good or bad, how many $k$-fractals are contained in $S$ and its remainder $\mathcal{R}(S, I)$ and whether they are grouped.

Attempting to cross good or bad columns will bring different challenges and we may require different features for the starting set $S$.
Below we introduce two quantities ($u_k$ and $v_k$) that will give a precise meaning to what we expect to happen in each of these situations.
The quantities $u_k$ and $v_k$ refer to the crossing of a good and a bad column respectively and they are central to the proof of Theorem~\ref{t:main}

Roughly speaking, Definition~\ref{d:u_k} introduces the first unlikely event: ``one tries to cross a good column, starting from a $k$-fractal set, but fails to obtain a remainder that is $k$-grouped and $k$-fractal''.

\begin{definition}
  \label{d:u_k}
  For $k \geq 0$, we introduce
  \begin{equation}
    \label{e:u_k}
    u_k(p) = \sup_{S, m, \xi^x, \xi^y} \mathbb{P}^{\xi^x, \xi^y}_p \Big[
    \text{$\mathcal{R}_k (S, I^x_m)$ does not contain a $k$-grouped, $k$-fractal set}
    \Big],
  \end{equation}
  where the suppremum is taken over all sets $S\subseteq \mathbb{Z}$, indices $m \in M_k$ and $\xi^x, \xi^y$ for which
  \begin{enumerate}[\qquad a)]
  \item {\bf ``one starts from a $k$-fractal'':} that is, $S$ and $\xi^y$ are such that $S$ forms a $k$-fractal.
  \item {\bf ``the column we are trying to cross is good'':} i.e. $\xi^x$ is such that $H^x_m = 0$;
  \end{enumerate}
  See Figure \ref{f:remainder}.
\end{definition}

\begin{remark}
  \label{r:iterate_uk}
  Note that the suppremum in the definition of $u_k$ runs over $k$-fractal sets $S$ that are not necessarily $k$-grouped.
  Much to the contrary, they could be very spread out, which is to be expected in case $S$ was obtained after crossing a very difficult column.
  In what follows, we show that when $p$ is close to one, $u_k$ decreases fast with $k$, which helps us cross good columns.
\end{remark}

We will now introduce the quantity $v_k$, corresponding to the probability of the second unlikely event.
Roughly speaking, it controls the probability that: ``one has attempted to cross an $h$-bad column with a sufficiently large initial set, but failed to obtain a $k$-fractal as a remainder''.

\begin{definition}
  \label{d:v_k}
  For $k \geq 0$, let
  \begin{equation}
    \label{e:v_k}
    v_k(p) = \sup_{h, S, m, \xi^x, \xi^y} \mathbb{P}^{\xi^x, \xi^y}_p \Big[
    \mathcal{R}_k(S, I^x_m) \text{ does not contain a $k$-fractal set}
    \Big],
  \end{equation}
  where the suppremum is taken over all sets $S\subseteq \mathbb{Z}$, indices $m \in M_k$, intensities $h \geq 1$ and environments $\xi^x, \xi^y$ for which
  \begin{enumerate}[\qquad a)]
  \item {\bf ``the column to be crossed has defect $h$'':} i.e. $\xi^x$ is such that $H^x_m = h$;
  \item {\bf ``$S$ is composed of a large family of $k$-fractals'':} $\xi^y$ is such that $S = S_1 \cup \cdots \cup S_{2^\ell}$, where $\ell = {\expdefect(h-1)}$ and $\{S_1, \ldots, S_{2^\ell}\}$ is a $k$-ordered family of $k$-fractal sets.
  \end{enumerate}
  See Figure \ref{f:remainder}.
\end{definition}

\begin{remark}
  \label{r:loss_v_k}
  The intuition behind the definition is that: ``given a defective column with intensity of defect $h \geq 1$, if one tries to traverse it starting from a family of $2^{\expdefect(h-1)}$ fractals, typically the remainder still contains at least one $k$-fractal''.
  If we compare the cardinality of the starting set ($2^{\expfractal k + \expdefect(h-1)}$) to that of  its remainder ($\geq 2^{\expfractal k }$), at most a fraction of $2^{\expdefect(h-1)}$ is lost due to the presence of the defect.
  In other words, the intensity of the defect controls how much we expect to divide the size of $S$ as it crosses the defective column.
\end{remark}

\begin{remark}
Let $\xi^x$ and $\xi^y$ be given.
If $I^x_m$ is a $k$-interval, $H^x_{m}\in \{0,1\}$ and $S$ is a $k$-fractal set which is $k$-grouped, then \eqref{e:remainder_box_A}, \eqref{e:u_k} and \eqref{e:v_k} imply
\begin{equation}
  \label{e:max_v_k_u_v}
  \mathbb{P}_p^{\xi^x, \xi^y} \big(\text{$\mathcal{R}_k(S, I^x_{m})$ does not contain a $k$-grouped, $k$-fractal set} \big) \leq \max\{u_k(p), v_k(p)\}.
\end{equation}
Since we will bound $u_k$ and $v_k$ by the same quantity, the above may be roughly understood as: provided we start from a $k$-regrouped $k$-fractal set it is not much harder to cross a column with a defect $H = 1$ than it is to cross a good column.
\end{remark}
\medskip

The heart of the proof of our main Theorem~\ref{t:main}, is to control the quantities $u_k$ and $v_k$.
More precisely, the next lemma contains the majority of our work and it shows that, for $p$ sufficiently close to one, $u_k$ and $v_k$ converge to zero quickly.

\begin{lemma}
  \label{l:summability}
  If $p > (1 - \scalezero^{ -10 })^{1/4}$ then for every $k \geq 0$
  \begin{eqnarray}
    \label{e:summability}
    \max\{u_k(p), v_k(p)\} \leq \scalezero^{ - k - 10}.
  \end{eqnarray}
\end{lemma}
The proof of Lemma \ref{l:summability} is presented in Section~\ref{s:contraction} below.

\section{Proof of decay for $u_k$ and $v_k$}
\label{s:contraction}

The objective of this section is to prove Lemma~\ref{l:summability}, which will follow from Lemmas~\ref{l:grouping} and~\ref{l:battle}.
Intuitively speaking these auxiliary lemmas give inductive bound on $u_{k+1}(p)$ and $v_{k+1}(p)$, respectively, in terms of $\max\{u_k(p), v_k(p)\}^2$.
For $p$ sufficiently close to one, this shows that $u_k(p)$ and $v_k(p)$ contract towards zero, ultimately leading to \eqref{e:summability}.
We start with some auxiliary preparation.

\subsection{Crossing multiple columns}
\label{s:multiple}

The definitions of $u_k$ and $v_k$ help us understand the effect of traversing a single good or bad column respectively.
But in our inductive argument, we will regard a column of scale $k + 1$ as being composed of $\scalezero$ columns of scale $k$.

This section is therefore dedicated to auxiliary lemmas that analyze the cumulative effect of traversing many columns, one after the other.
This motivates the definition of a good block below, which is reminiscent of that of a good interval \eqref{e:good_single_bad}.

\begin{definition}[$k$-blocks]
  \label{d:good_blocks}
  We define a \emph{$k$-block} as a union of consecutive $k$-intervals, that is $I = \cup_{i=i_0}^{i_1} I_{(k,i)}$.
  Moreover
  \begin{enumerate}[\quad a)]
  \item we define the \emph{weight} $H_{(k, I)}$ of $I$ as $H_{(k, i_0)} + H_{(k, i_0 + 1)} + \dots + H_{(k, i_1)}$;
  \item we say that $I$ is $k$-good if $H_{(k, I)} \leq 1$ and
  \item we introduce the \emph{length} $\ell_k(I)$ as being $i_1 - i_0$.
  \end{enumerate}
\end{definition}

We note that we only deal with blocks of length $\ell_k(I) \leq L$, which is enough to bridge between scales $k$ and $k + 1$.
In case $i_0 = j \scalezero$ and $i_1 = (j+1) \scalezero -1$ for some $j \in \mathbb{Z}$ the good $k$-block is indeed a good interval at scale $k+1$.
\medskip

We introduce below the notion of corridors, which are paths in our renormalized lattice.

\begin{definition}[Corridors]
  \label{d:corridor}
  A (horizontal) $k$-corridor is a path $\mathcal{C} = \big( (i_0, j),$ $(i_0 + 1, j),$ $\dots, (i_1, j) \big)$ in $\mathbb{Z}^2$.
  Moreover, given $\xi^x$ and $\xi^y$, we say that such a horizontal corridor $\mathcal{C}$ is \emph{good} if
  \begin{itemize}
  \item $I^y_{(k, j)}$ is a good $k$-interval and
  \item the $k$-block $I = \cup_{i = i_0}^{i_1} I_{(k, i)}$ is $k$-good.
  \end{itemize}
  Finally, given a horizontal $k$-corridor $\mathcal{C}$ as above and a $k$-fractal set $S\subseteq I^y_{(k, j)}$, we say that $\mathcal{C}$ is \emph{well-crossed} by $S$ if the remainder $\mathcal{R}_k(S, I)$ contains a $k$-fractal set.
  We define (vertical) $k$-corridors and when they are well crossed analogously.
\end{definition}

The fact that $S \subseteq I^y_m$ in Definition \ref{d:corridor} guarantees that
$\mathcal{R}_k(S,I) \subseteq I^y_m$, hence every $k$-fractal set in $\mathcal{R}_k(S,I)$ is automatically $k$-grouped (see  \eqref{e:remainder_box_A}).
The following lemma relates the probability of crossing a $k$-corridor in terms of the quantities $u_k$ and $v_k$.

\begin{lemma}[Traversal of good corridors and blocks]
  \label{l:corridor}
  If $\mathcal{C} = \big( (i_0, j), \dots, (i_1, j) \big)$ is a good $k$-corridor, then for every $k$-fractal set $S \subseteq F^l(R_{(k, i_0, j)})$ we have
  \begin{eqnarray}
    \label{e:prob_corridor_traversed}
    \mathbb{P}_p^{\xi^x, \xi^y} \big( \mathcal{C} \text{ is not well crossed by $S$} \big) \leq (i_1 - i_0 + 1) \max\{u_k(p), v_k(p)\}.
  \end{eqnarray}
  Moreover, if $I=\cup_{i=i_0}^{i_1}I_{(k,i)}$ is a good $k$-block with $i_1-i_0\geq 1$ and $S$ is any $k$-fractal (not necessarily $k$-grouped), we have
  \begin{eqnarray}
    \label{e:several_group}
    \mathbb{P}^{ \xi^x, \xi^y }_p
    \Bigg(  \displaystyle{
    \begin{split}
      \textrm{$\mathcal{R}_k \big(S, I\big)$ does not contain \,\,\,} \\
      \text{ any $k$-grouped, $k$-fractal set }
    \end{split}}
    \Bigg)
    \leq (i_1 - i_0+1) \max\{u_k(p), v_k(p)\}.
  \end{eqnarray}
\end{lemma}

\begin{figure}[!htb]
  \centering
  \includegraphics[width=.9\textwidth]{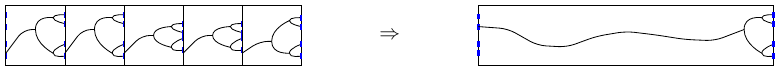}
  \caption{We illustrate a $k$-corridor $R = I \times I^y_m$, where $I$ is a $k$-block. If all the $k$-intervals the form $I$ are well crossed, then $I$ is well crossed.}
  \label{f:grouping}
\end{figure}

\begin{proof}
  Note that
  \begin{equation}
    \label{eq:1}
    \mathbb{P}_p^{\xi^x, \xi^y} \big( \mathcal{C} \text{ is not well crossed} \big) = \mathbb{P}_p^{\xi^x, \xi^y} \big( \mathcal{R}_k(S, I) \text{ does not  contain a $k$-fractal set}\big)\\
  \end{equation}
  Let $I=\cup_{i=i_0}^{i_1} I_{(k,i)}$ and let $\mathfrak{F}$ denote the set of subsets of $I^y_m$ that contain a $k$-fractal set.
  Define $S^{(i_0)}=S$, which by assumption belongs to $\mathfrak{F}$.
  \begin{equation}
    \label{e:iterate_u_k}
    \begin{split}
      \mathbb{P}^{\xi^x, \xi^y}_p \big( \mathcal{R}_k(S, I) &\text{ contains a $k$-fractal set}\big) \\
 & \geq  \sum_{S^{(i_0+1)},\ldots ,S^{(i_1)} \in \mathfrak{F}}  \mathbb{P}^{\xi^x, \xi^y}_p \Big( \bigcap_{i = i_0}^{i_1-1} \Big\{S^{(i + 1)} = \mathcal{R}\big(S^{(i)}, I^x_{(k,i)}\times I^y_{m}\big) \Big\} \Big).
    \end{split}
  \end{equation}
  The events appearing in the intersection on the right-hand side in \eqref{e:iterate_u_k} are independent because they are measurable with respect to the percolation process inside disjoint rectangles of type $I^x_{(k,i)}\times I^y_{m}$.
  Therefore, for a fixed choice for $S^{(i_0+1)}, \ldots, S^{(i_1)}$ we have
  \begin{equation}
    \begin{split}
      \mathbb{P}^{\xi^x, \xi^y}_p  \Big( \bigcap_{i = i_0}^{i_1-1} \Big\{S^{(i + 1)} = \mathcal{R}\big(S^{(i)}, I^x_{(k,i)}\big) \Big\} \Big) =
      \prod_{i = i_0}^{i_1-1} \mathbb{P}^{\xi^x, \xi^y}_p \Big(S^{(i + 1)}= \mathcal{R}\big(S^{(i)}, I^x_{(k,i)}\big)\Big)
    \end{split}
  \end{equation}
  Summing over $S^{(i)} \in \mathfrak{F}$ we obtain
  \begin{equation}
    \begin{split}
      \sum_{S^{(i)} \in \mathfrak{F}}
      & \mathbb{P}^{\xi^x, \xi^y}_p \Big(  S^{(i)} = \mathcal{R}\big(S^{(i-1)}, I^x_{(k,i-1)}\big)  \Big) \\
      & = \mathbb{P}^{\xi^x, \xi^y}_p \Big(\mathcal{R}\big(S^{(i-1)}, I^x_{(k,i-1)}\big) \textrm{ contains a $k$-fractal set}\Big) \overset{\eqref{e:max_v_k_u_v}}{\geq} 1 - \max\{u_k(p), v_k(p)\}.
    \end{split}
  \end{equation}
  Iterating, we obtain
  \begin{equation}\label{e:final_proof_corridor}
    \begin{split}
      \sum_{S^{(i_0+1)}\cdots S^{(i_1)} \in \mathfrak{F}}  \mathbb{P}^{\xi^x, \xi^y}_p & \Big( \bigcap_{i = i_0}^{i_1-1} \Big\{S^{(i + 1)} = \mathcal{R}\big(S^{(i)}, I^x_{(k,i)}\times I^y_{m}\big) \Big\} \Big)
      \\
    & \geq \big(1 - \max \big\{ u_k(p), v_k(p) \big\} \big)^{i_1 - i_0+1}  \geq  1 - (i_1 - i_0 + 1) \max \{ u_k(p), v_k(p) \}.
    \end{split}
  \end{equation}
  Plugging this estimate into \eqref{e:iterate_u_k} yields the desired result.

  For the second part, note that $H^x_{(k,i_0)} + H^x_{(k,i_0+1)} \leq 1$, so we argue that
  \begin{eqnarray}
    \label{e:initial_group}
    \mathbb{P}^{ \xi^x, \xi^y }_p
    \Bigg(  \displaystyle{
    \begin{split}
      \textrm{$\mathcal{R}_k \big(S, I^x_{ (k, i_0)} \cup I^x_{ (k, i_1)}\big)$ does not\,\,\,\,\,\,\,\,\,\,} \\
      \text{contain a $k$-grouped, $k$-fractal set }
    \end{split}}
    \Bigg)
    \leq 2\max\{u_k(p), v_k(p)\}.
  \end{eqnarray}
  To see why this holds, note that one of three cases occurs: either $H_{(k, 0)}=H_{(k, 1)}=0$; or $H_{(k, 0)}=0$ and $H^x_{(k, 1)}=1$; or $H^x_{(k, 0)}=1$ and $H^x_{(k, 1)}=0$.
  The three cases can be treated similarly and we only deal with the third one.
  If the event displayed in \eqref{e:initial_group} occurs, then either $\mathcal{R}_k \big( S,  I^x_{(k,0)} \big)$ fails to contain a $k$-fractal set, or else it does contain such a $k$-fractal set $S'$ but $\mathcal{R}_k \big(S',  I^x_{(k,1)} \big)$ does not contain a $k$-grouped $k$-fractal set.
  The former event has probability at most $v_k(p)$ and the latter at most $u_k(p)$, so the total probability is bounded by $u_k(p)+v_k(p) \leq 2 \max\{u_k(p),v_k(p)\}$.

  Supposing that $\mathcal{R}_k(S, I^x_{(k_0,i)}\cup I^x_{((k_0+1))})$ contains any $k$-fractal set in $I^y_{(k,j)}$, we consider the corridor $((i_0+2,j),\cdots, (i_1,j))$ and the result follows from \eqref{e:final_proof_corridor} and \eqref{e:initial_group}.
\end{proof}

The next result is purely geometrical and shows how we can use crossings at scale $k$ in order to connect $k$-fractals to $(k + 1)$-fractals.

\begin{lemma}
  \label{l:branch}
  For $k \geq 0$, fix a rectangle $R = I^x \times I^y$ where $I^y$ is an interval at scale $k + 1$ (that is, $I^y = I^y_m$ for some $m \in M_{k + 1}$) and $I^x$ is a $k$-block with length at least $L/2$ (or more precisely, $I^x = \cup_{i = i_0}^{i_1} I^x_{(k, i)}$ with $i_1 - i_0 \geq L/2$).
  Then there exist
  \begin{itemize}
  \item $\scalezero$ disjoint horizontal corridors ($\mathcal{C}_1, \dots, \mathcal{C}_\scalezero$) traversing $R$ horizontally and
  \item four disjoint vertical corridors traversing $R$ vertically,
  \end{itemize}
  such that every crossing of a horizontal corridor intersects every crossing of a vertical corridor.
\end{lemma}

\begin{proof}
Let $I^y_m=I^y_{(k+1,j)}=\cup_{l=0}^{L-1} I^y_{(k,jL+l)}$.
The horizontal corridors are defined by
\begin{eqnarray*}
\mathcal{C}_l=((i_0,jL+l-1),...,(i_1,jL+l-1)) \textrm{ for } l=1,... , L.
\end{eqnarray*}
For vertical corridors, choose four $k$-intervals in $I^x$ and take their cartesian product with $I^y$.
\end{proof}

\begin{remark}
  \label{r:parallel}
  Observe that the horizontal corridors $\mathcal{C}_1, \dots, \mathcal{C}_\scalezero$ in Lemma~\ref{l:branch} are given by vertical translations of one another.
  Moreover the left face of these corridors pave the interval $I^y$, meaning that every $k$-fractal, $k$-grouped set in $I^y$ is contained in the face of one such corridor.
\end{remark}

The geometric result above gives us an intuition of how we could use the corridors of scale $k$ to connect a $k$-fractal into a $(k + 1)$-fractal, effectively bootstrapping our scale.
This is what we describe as a ``recovery'', which we define precisely in the following.

\begin{definition}
  \label{d:recovery}
  Given a $k$-fractal set $S$ and an interval $I$ contained in a $(k + 1)$-interval, we define the event
  \begin{equation}
    \label{e:recovery}
    \big\{ S \text{ recovers in $I$} \big\}
    := \bigg\{
    \begin{gathered}
      \mathcal{R}_{k + 1}(S, I) \text{ contains a $k$-ordered family with}\\
      \text{$(L - 1)$ $k$-fractal sets in a $(k + 1)$ interval}
  \end{gathered}
    \bigg\},
  \end{equation}
  see Figure~\ref{f:recovery} for some illustrations of the above.
\end{definition}

\begin{remark}
  \label{r:recovery}
  Requiring that the remainder contains a $k$-ordered family composed of $(\scalezero - 1)$ $k$-fractal sets suffices to guarantee that it contains a $(k + 1)$-fractal set, or in other words
  \begin{equation}
    \label{e:recovery_k+1}
    \Big\{ \text{$S$ recovers in $I$} \Big\} \subseteq \Big\{ \mathcal{R}_{k + 1}(S, I) \text{ contains a $(k + 1)$-fractal set}\\ \Big\}.
  \end{equation}
  In fact only $2^{\expfractal}$ such sets would be enough.
  The main reason why we require $(L - 1)$ many $k$-fractal sets instead of one $(k + 1)$-fractal will become clear in \eqref{e:traversal_before_defect} below.
\end{remark}

\begin{figure}[h]
  \centering
  \includegraphics[width=0.6\textwidth]{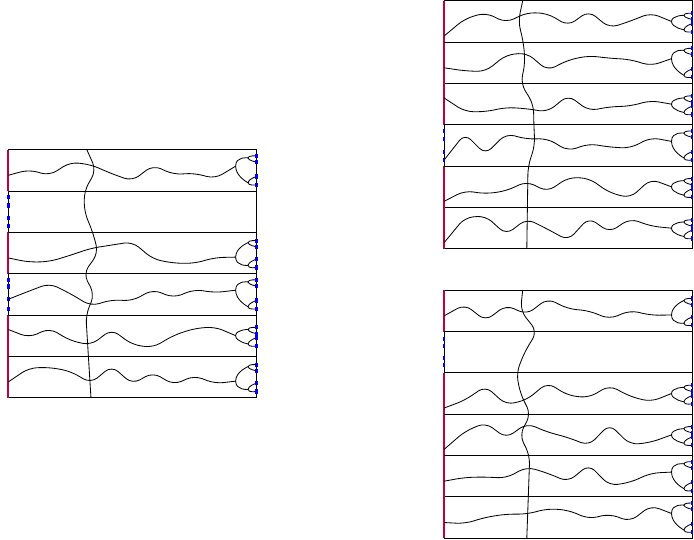}
  \caption{Consider $S'$ and $S''$, two $k$-grouped $k$-fractal sets at the left edge of a column.
    {\bf Left -} if $S'$ and $S''$ belong to the same $(k + 1)$-interval.
    {\bf Right -} if on the other hand $S'$ and $S''$ belong to different $(k + 1)$-intervals.
    {\bf Observation -} in any of the above cases, if all but one of the horizontal corridors are crossed, then in the presence of a single vertical crossing in the box interconnect numerous $k$-grouped $k$-fractal sets on the opposite side (represented in blue on the right face).
  }
  \label{f:recovery}
\end{figure}

Having defined what we mean by ``recovery'' we are ready to estimate its probability.
Roughly speaking, the next lemma derives an upper bound on the probability of not recovering from $k$-fractals to a $(k + 1)$-fractal in terms of $\max \{u_k, v_k\}^2$ (which is vital to prove that these quantities contract to zero).

\begin{lemma}[Recovery]\label{l:recovery}
  Fix $\xi^x$ and $\xi^y$ and let $\{S_1,S_2,S_3\}$ be a $k$-ordered family of $k$-fractal sets.
  Suppose further that $\xi^y$ is such that every interval in $Z_{k+1}(S_1 \cup S_2 \cup S_3)$ is good.
  If $I$ is a good $k$-block for $\xi^x$ with $\ell{(I)} \in [\scalezero/2 + 2, \scalezero]$, then
  \begin{equation}
    \label{e:prob_recovery}
    \begin{split}
      \mathbb{P}^{\xi^x,\xi^y}_p \Big(
      \text{none of $S_1$, $S_2$ and $S_3$ recovers in $I$}
      \Big)
      & \leq \scalezero^5 \big(\max\{u_k(p),v_k(p)\}\big)^2.
    \end{split}
  \end{equation}
\end{lemma}

In view of \eqref{e:recovery_k+1}, for $S_1$, $S_2$, $S_3$ and $I$ as in the statement of Lemma \ref{l:recovery}, we have
\begin{equation}
  \label{e:recovery}
  \begin{split}
    \mathbb{P}^{\xi^x,\xi^y}_p \Bigg( \bigcap_{i=1,2,3}
    \bigg\{
    \begin{gathered}
      \text{$\mathcal{R}_{k+1}(S_i, I)$ does not contain}\\
      \text{a $(k+1)$-fractal set}
    \end{gathered}
    \bigg \} \Bigg)  \leq \scalezero^5 \big(\max\{u_k(p),v_k(p)\}\big)^2,
  \end{split}
\end{equation}
which shows the high probability of recovery in good parts of the environment.

\begin{proof}
  Without loss of generality suppose $i_0=0$.

  The proof has two main steps that we informally describe as ``regrouping'' and ``branching'':
  \begin{itemize}
  \item ``regrouping'': first we will look at the remainders of $S_1$, $S_2$ and $S_3$ after traversing the two left-most columns, $I^x_{(k, 0)}$ and $I^x_{(k, 1)}$, and try to find at least one pair (say $S'$ and $S''$) of $k$-grouped $k$-fractal sets;
  \item ``branching'': then, assuming that we have succeeded in the previous step, we will use Lemma~\ref{l:branch} once or twice, depending on whether $S'$ and $S''$ belong to the same $k + 1$ box or not, see Figure~\ref{f:recovery} for an illustration of the two cases.
  \end{itemize}
  Turning to the rigorous proof, note first that for each $i\in\{1, 2, 3\}$, \eqref{e:initial_group} readily implies that
  \begin{eqnarray}
    \mathbb{P}^{ \xi^x, \xi^y }_p
    \Bigg(  \displaystyle{
    \begin{split}
      \textrm{$\mathcal{R}_k \big(S_i, I^x_{ (k, 0)} \cup I^x_{ (k, 1)}\big)$ does not\,\,\,\,\,\,\,\,\,\,} \\
      \text{contain a $k$-grouped, $k$-fractal set }
    \end{split}}
    \Bigg)
    \leq 2\max\{u_k(p), v_k(p)\}.
    \label{e:fail_group_2_columns}
  \end{eqnarray}
  Notice that for different indices $i \in \{1,2,3\}$, the events in \eqref{e:fail_group_2_columns} are independent (due to our assumption that the $S_i$'s are $k$-ordered).
  Therefore, we estimate the probability of observing at least two occurrences of such events as follows
  \begin{equation}
    \label{e:initial_group_2}
    \begin{split}
      \mathbb{P}^{ \xi^x, \xi^y }_p
      \Bigg(
      & \displaystyle{
        \begin{split}
          \textrm{for two indices $i \in \{1,2,3\}$, the set $\mathcal{R}_k \big(S_i, I^x_{ (k, 0) } \cup I^x_{ (k, 1) } \big)$} \\
          \textrm{ contains a $k$-grouped, $k$-fractal set \qquad \qquad}
        \end{split}}
        \Bigg)\\
      & \geq
        1 - \mathbb{P}^{ \xi^x, \xi^y }_p
        \Bigg(  \displaystyle{
        \begin{split}
          \textrm{for two indices $i \in \{1,2,3\}$, the set $\mathcal{R}_k \big(S_i, I^x_{ (k, 0) } \cup I^x_{ (k, 1) } \big)$} \\
          \textrm{ does not contain a $k$-grouped, $k$-fractal set \qquad \qquad}
        \end{split}}
        \Bigg)\\
      & \geq 1 - \binom{3}{2} 2 \big(\max\{u_k(p), v_k(p)\}\big)^2 = 1 - 12\big(\max\{u_k(p), v_k(p)\}\big)^2.\\
    \end{split}
  \end{equation}

  We now move to the second step.
  Conditioning that the event displayed in the left hand side of \eqref{e:initial_group_2} occurs, we have that $\mathcal{R}_k(S_1 \cup S_2 \cup S_3, I^x_{ (k, 0)}\cup I^x_{(k,1)})$ contains two $k$-grouped, $k$-fractal sets $S'$ and $S''$.
  Denote $I'=\cup_{j=2}^{i_1}I^x_{(k,j)}$ the portion of $I$ that still remains to be traversed, which we will attempt to do starting from $S'$ and $S''$.
  Note that $\ell(I') = \ell(I) - 2 \geq L/2$.

  We now consider two cases depending on whether $S'$ and $S''$ lie either close or far apart from each other.

  \medskip

  \noindent {\bf Case 1 -} ($Z_{k+1}(S') \neq Z_{k+1}(S'')$:
  In this case, we know that $S'$ and $S''$ belong to disjoint $k + 1$ intervals and we write $S$ for either of them.
  Lemma~\ref{l:branch} gives us four vertical corridors and $\scalezero$ horizontal ones and according to Remark~\ref{r:parallel}, $S$ is contained in the left face of one of these horizontal corridors.
  Assuming without loss of generality that $S \subseteq F^l(\mathcal{C}_1)$,
  \begin{display}
    \label{e:geometric_recover_0}
    if $\mathcal{C}_1$ is well-crossed by $S$ and all other corridors are well-crossed (by arbitrary $k$-fractals), then $\mathcal{R}_{k + 1}(S, I^x)$ contains a $k$-ordered family with $\scalezero$ $k$-fractals.
  \end{display}
  See the top-right picture in Figure~\ref{f:recovery}.
  This implies the following estimate:
  \begin{equation}
    \label{e:recover_spread}
    \begin{split}
      & \mathbb{P}^{\xi^x,\xi^y}_p \Big( \bigcap_{S = S', S''} \big\{ \mathcal{R}_{k+1}(S, I') \textrm{ does not contain $\scalezero$ $k$-ordered, $k$-fractal sets} \big \} \Big)\\
      & \leq
        \mathbb{P}^{\xi^x,\xi^y}_p
        \Bigg(\;
        \begin{split}
          \text{one of the horizontal or one of the vertical}\\
          \text{corridors in Lemma~\ref{l:branch} is not well-crossed}
        \end{split}
        \;\Bigg)^2 \\
      & \leq
        \Big( 2^{16} \scalezero \max \{ u_k(p), v_k(p) \} \Big)^2
        \leq L^4 \big( \max \{ u_k(p), v_k(p) \} \big)^2.
    \end{split}
  \end{equation}

  \medskip

  \noindent {\bf Case 2 -} ($Z_{k+1}(S') = Z_{k+1}(S'')$:
  In this case, we know that $S'$ and $S''$ belong to the same $k + 1$ interval and we use again Lemma~\ref{l:branch} to obtain $\scalezero$ horizontal corridors and four vertical ones.
  As stated in Remark~\ref{r:parallel}, $S'$ and $S''$ must be contained in the left face of two distinct such corridors.
  Assuming without loss of generality that $S' \subseteq F^l(\mathcal{C}_1)$ and $S'' \subseteq F^l(\mathcal{C}_2)$,
  \begin{display}
    \label{e:geometric_recover}
    if all but one horizontal corridor are well-crossed ($\mathcal{C}_1$ by $S'$, $\mathcal{C}_2$ by $S''$ and the others by arbitrary sets) and all but one vertical corridor are well-crossed, then $\mathcal{R}_{k + 1}(S' \cup S'', I^x)$ contains a $k$-ordered family with $(\scalezero - 1)$ $k$-fractal sets.
  \end{display}
  See the left hand side of Figure~\ref{f:recovery}.
  This yields:
  \begin{equation*}
    \begin{split}
      & \mathbb{P}^{ \xi^x,\xi^y }_p
        \Big(
        \mathcal{R}_{k+1}( S' \cup S'', I' ) \text{ does not contain $(\scalezero - 1)$ $k$-ordered, $k$-fractal sets} \Big)\\
      & \leq
        \mathbb{P}^{ \xi^x,\xi^y }_p \bigg(
        \begin{split}
          \big\{
          \text{two horizontal corridors are not well-crossed}
          \big\}\\
          \; \mcup \;
          \big\{
          \text{two vertical corridors are not well-crossed}
          \big\}
        \end{split}
        \;
        \bigg) \\
      & \overset{\eqref{e:prob_corridor_traversed}}\leq
        \big( 2^{16} \scalezero \max \{ u_k(p), v_k(p) \} \big)^2 + \big( 4 \scalezero \max \{ u_k(p), v_k(p) \} \big)^2 \leq
        2 \scalezero^4 \big( \max\{u_k(p), v_k(p)\} \big)^2,
    \end{split}
  \end{equation*}
  see left hand side of Figure~\ref{f:recovery}.

  \medskip

  The above bound, together with \eqref{e:recover_spread} allow us to conclude from \eqref{e:initial_group_2} that
  \begin{equation}\label{e:final_proof_recovery}
    \begin{split}
      \mathbb{P}^{\xi^x,\xi^y}_p \Bigg( \bigcap_{i=1,2,3}
      & \bigg\{
      \begin{gathered}
        \text{$\mathcal{R}_{k+1}(S_i, I)$ does not contain a $k$-ordered}\\
        \text{ family of $(\scalezero - 1)$ $k$-fractal sets}
      \end{gathered}
      \bigg \} \Bigg) \\
      & \leq 12 \big( \max\{u_k(p), v_k(p)\}\big)^2 + 2 \scalezero^4 \big( \max\{u_k(p), v_k(p)\}\big)^2 \\
      & \leq \scalezero^5 \big( \max\{u_k(p), v_k(p)\}\big)^2,
    \end{split}
  \end{equation}
  finishing the proof.
\end{proof}

We now state a simple concentration bound for Bernoulli random variables whose proof we include for the readers' convenience.
\begin{lemma} Let $\{X_i\}_{i\in\mathbb{Z}}$ be i.i.d.\ random variables with $\mathbb{P}(X_0 =1) = \alpha = 1-\mathbb{P}(X_0=0)$ for some $\alpha\geq 0.9$.
Then for every $n\geq 1$ we have
\begin{equation}\label{e:concentration}
 \mathbb{P} \bigg( \sum_{i=1}^n X_{i} \leq  n/2 \bigg) \leq 10 (1-\alpha).
\end{equation}
\end{lemma}

\begin{proof}
  We have
  \begin{equation}
    \begin{split}
      &\mathbb{P} \bigg( \sum_{i=1}^n X_{i} \leq   n/2  \bigg) =
        \mathbb{P} \bigg( \sum_{i=1}^n ( \alpha - X_{i}) \geq  (\alpha- 1/2) n \bigg)
        \leq \mathbb{P} \bigg( \Big | \sum_{i=1}^n ( X_{i} - \alpha ) \Big| \geq 2n/5  \bigg) \\
      & =\mathbb{P} \bigg( \Big | \sum_{i=1}^n ( X_{i} - \alpha ) \Big|^2 \geq  4n^2/25\bigg) \leq \dfrac{ 25 n\alpha( 1 - \alpha ) }{4 n^2} \leq 10 (1-\alpha),
    \end{split}
  \end{equation}
  where, in the last two steps we have first used Markov's inequality and then the fact that $\alpha<1$ and $n\geq 1$.
\end{proof}

We will also be interested in crossing harder defects (with arbitrary $h$).
Recall from Remark~\ref{r:loss_v_k} that, when crossing a column with intensity $h$, the factor by which the cardinality of $S$ drops will be related with the value of $h$.
More precisely, the definition of $v_k$ suggests one loses a fraction of $2^{\expdefect(h-1)}$ points as we traverse a defect of intensity $h$.
The next lemma extends this description to a bad column that is surrounded by good ones.

\begin{lemma}[Traversing defects]
  \label{l:algebra}
  Let $I = \cup_{i = i_0}^{i_1} I^x_{(k,i)}$ be a union of consecutive $k$-intervals.
  Suppose that
  \begin{itemize}
  \item there exists $i'\in [i_0, i_1]$ such that $H_{(k,i)}= h >0$,
  \item $H_{(k,i)}= 0$  for all $i\in [i_0, i_1]\setminus\{i'\}$ and
  \item $0\leq i_1-i_0\leq \scalezero-1$.
  \end{itemize}
  If $S$ is any $k$-ordered family composed of $2^{l}$ $k$-fractal sets, with $l\geq \expdefect  h $,  then
  \begin{equation}
    \begin{split}
      \label{e:lemma_algebra}
      \mathbb{P}^{\xi^x,\xi^y}_p
      \Bigg(
      \displaystyle{
      \begin{split}
        \mathcal{R}_{k}(S, I) \textrm{ does not contain a $k$-ordered } \\
        \textrm{ family with $2^{l -\expdefect   h }$ $k$-fractal sets \,\,\,\,\,\,\,\,\,\,\, }
   \end{split}
      }
      \Bigg)
      \leq \scalezero^5 \big( \max\{ u_k(p),v_k(p) \} \big)^2.
    \end{split}
  \end{equation}
\end{lemma}

\begin{figure}[h]
  \centering
  \includegraphics[scale=.7]{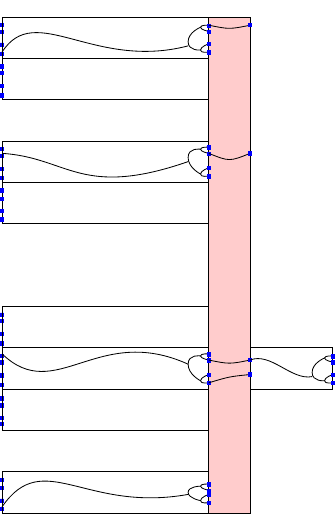}
  \caption{The crossing of $S$ along $I$ will be done in three parts. First, we cross the good region of $I$ until we reach the defect, if it exists, using $u_k$. Then, we cross the defect using $v_k$. And finally, we cross the remaining good region using $u_k$ again.}
  \label{f:three_fases}
\end{figure}

\begin{proof}
  We first write $S = S_1 \cup S_2$ where for each $i = 1, 2$ we have that $S_i=S_{i,1} \cup \cdots \cup S_{d,2^{l-1}}$ is a $k$-ordered family of $2^{l-1}$ $k$-fractal sets.
  Using independence of the percolation inside each strip $ I \times S_i$ we obtain that the left-hand side in \eqref{e:lemma_algebra} is bounded above by
  \begin{equation}
    \prod_{i=1,2} \mathbb{P}^{\xi^x,\xi^y}_p
    \Bigg(
    \displaystyle{
      \begin{split}
        \mathcal{R}_{k}(S_i, I) \textrm{ does not contain a $k$-ordered } \\
        \textrm{ family of $2^{l -\expdefect h }$ $k$-fractal sets\ \,\,\,\,\,\,\,\,\,\,\, }
      \end{split}
    }
    \Bigg).
    \label{e:division_concentration}
  \end{equation}
  The preceding product is responsible for the square appearing on the right-hand side in \eqref{e:lemma_algebra}.

  The rest of the argument is divided into estimates on the probability of crossings the three different regions, see Figure~\ref{f:three_fases}.

  \noindent \textbf{Part 1 - crossing the $k$-block $I_1=\cup_{i=i_0}^{i'-1} I^x_{(k,i)}$:}
  Note that $I_1$ is a union over good $k$-intervals only.
  Recall that $S_1=S_{1,1} \cup \cdots \cup S_{1,2^{l-1}}$ is a $k$-ordered family of $2^{l-1}$ $k$-fractal sets.
  For $j=1,\ldots, 2^{l-1}$, denote
  \begin{equation}
    X_{j} := \mathbbm{1}\big\{\{\mathcal{R}_k(S_{1,j},I_1)\textrm{ contains a $k$-fractal set}\}\big\},
  \end{equation}
  which are independent random variables.

  Let $\alpha:= 1-\scalezero \, u_k(p)$.
  We can assume that $\alpha>0.9$, since if this was not the case we would have $\scalezero^5(\max\{u_k(p), v_k(p)\})^2 \geq 1$ and \eqref{e:lemma_algebra} would hold immediately.

  For each fixed $j \in 1,\ldots, 2^{l-1}$ we can apply \eqref{e:several_group} in order to obtain
  \begin{equation}
    \mathbb{P}_p^{\xi^x, \xi^y}( \mathcal{R}_k(S_{1,j},I_1)\textrm{ contains a $k$-fractal set} ) \geq 1 - (i'-i_0) u_k(p) \geq \alpha.
  \end{equation}
  Therefore the family $(X_{j})_{j}$ dominates an i.i.d.\ collection of random variables having the Bernoulli distribution with parameter $\alpha$.
  Applying \eqref{e:concentration} we have that
  \begin{equation}
    \label{e:first_good_columns}
    \begin{split}
      \mathbb{P}^{\xi^x, \xi^y}_p
      & \Big( \mathcal{R}_{k}(S_1, I_1) \textrm{ does not contain a $k$-ordered family of $2^{l - 2}$, $k$-fractal sets} \Big)\\
      \leq
      & \,\, \mathbb{P}^{\xi^x,\xi^y}_p \bigg( \sum_{j=1,\ldots,2^{l - 1}} X_{j} \leq   2^{ l - 2 }  \bigg)  \leq 10  (1 - \alpha) =10 \scalezero  u_k(p) \leq \scalezero^2  u_k(p),
    \end{split}
  \end{equation}
  and similarly for $S_2$.

  \medskip

  \noindent \textbf{Part 2 - crossing the $k$-block $I_2=I_{(k,i')}^x$:}
  On the event that $\mathcal{R}_{k}(S_1, I_1)$ contains a $k$-ordered family of $ 2^{ l - 2 } $ $k$-fractal sets, we can condition on the realization of these $k$-fractal sets and gather them into families $S'_{j} = S'_{j,1} \cup \cdots \cup S'_{j,2^{\expdefect (h-1)}}$, with $j=1, \ldots,  2^{ l  - \expdefect h +2} $.
  Each of these families is $k$-ordered and contains $2^{\expdefect (h-1)}$ $k$-fractal sets.
  For $j=1, \ldots, 2^{ l - \expdefect h + 2 }$, let
  \begin{equation}
    X'_j := \mathbbm{1}\big\{ \mathcal{R}_{k}( S'_{i,j}, I_2 )\textrm{ contains a $k$-fractal set}\big\},
  \end{equation}
  which are independent random variables.
  Using Definition \ref{d:v_k} one may show that they dominate stochastically a family of i.i.d.\ random variables having the Bernoulli distribution with parameter $\alpha':= 1 - v_k(p)$.
  An argument similar to the one employed in the preceding case shows that we can assume that $\alpha'\geq 0.9$.
  Hence, for $i=1,2$
  \begin{equation}
    \label{e:concentration_defect_1}
    \begin{split}
      \mathbb{P}^{\xi^x,\xi^y}_p
      & \Bigg(
        \displaystyle{
        \begin{split}
          \mathcal{R}_{k}(S_1, I_1\cup I_2) \textrm{ does not contain a } \\
          \textrm{  $k$-ordered family of $ 2^{l - \expdefect h + 1}$ $k$-fractal sets }
        \end{split}
        }
        \Bigg) \\
      & \leq \scalezero^2\,  u_k(p)   +\mathbb{P}^{\xi^x,\xi^y}_p \bigg( \sum_{ i = 1 }^{  2^{l - \expdefect h + 2} } X'_j \leq 2^{l - \expdefect h + 1}\bigg)\\
      & \leq \scalezero^2\, u_k(p) + 10  v_k(p) \leq( \scalezero^2 +10 ) \big(\max \{ u_k(p), v_k(p)\}\big),
    \end{split}
  \end{equation}
  and similarly for $S_2$.
  \medskip

  \noindent \textbf{Part 3 - crossing the $k$-block $I_3=\cup_{i=i'+1}^{i_1}I_{(k,i)}^x$:}
  On the event that $\mathcal{R}_{k}(S_1, I_1 \cup I_2)$ contains a $k$-ordered family of $\ 2^{ l - \expdefect h + 1 }$ $k$-fractal sets,  we can condition on their realization $S_{i,1}'', \cdots, S_{i, 2^{ l - \expdefect h + 1 } }'' $ and denote, for $j=1, \ldots,  2^{ l - \expdefect h + 1 } $,
  \begin{equation}
    X''_j := \mathbbm{1}\big\{\mathcal{R}_{k}( S''_{1,j}, I_3 )\textrm{ contains a $k$-fractal set}\big\}
  \end{equation}
  This family of random variables dominate stochastically a family of i.i.d.\ random Bernoulli random variables with parameter $\alpha= 1 - \scalezero u_k(p)$ (recall that we may assume $\alpha \geq 0.9$).
  Hence, we have:
  \begin{equation}
    \label{e:concentration_2}
    \begin{split}
      \mathbb{P}^{\xi^x,\xi^y}_p
      & \Bigg(
        \displaystyle{
        \begin{split}
          \mathcal{R}_{k}(S_1, I) \textrm{ does not contain a $k$-ordered } \\
          \textrm{ family of $2^{l - \expdefect h }$ $k$-fractal sets\,\,\,\,\,\,\,\,\,\,\,\, }
        \end{split}
        }
        \Bigg) \\
      & \leq  (\scalezero^2 + 10  ) \big(\max \{ u_k(p), v_k(p)\}\big)   +\mathbb{P}^{\xi^x,\xi^y}_p \bigg( \sum_{ i = 1 }^{  2^{l - \expdefect h + 1} } X''_{i,j} \leq 2^{l - \expdefect h } \bigg),\\
      & \leq  (\scalezero^2 + 10  )\big(\max \{ u_k(p), v_k(p)\}\big) + 10\scalezero  u_k(p) \leq 3 \cdot  \scalezero^2 \big(\max \{ u_k(p), v_k(p)\}\big)
    \end{split}
  \end{equation}
  Using \eqref{e:first_good_columns}, \eqref{e:concentration_defect_1} and \eqref{e:concentration_2} in \eqref{e:division_concentration}, we obtain the result.
\end{proof}

\subsection{Recursive inequalities for grouping and traversing}
\label{s:recursive}

In the previous sub-section, we have studied the effect of traversing several columns sequentially.
We can now regard a column at scale $k + 1$ as $\scalezero$ columns in scale $k$ and use the previous lemmas to write recursive inequalities for the quantities $u_k$ and $v_k$.
This is the main purpose of this section.
These inequalities will later be used to show that, for $p$ small, the sequences $u_k(p)$ and $v_k(p)$ converge fast to zero.
We begin with $u_{k+1}$.

\begin{lemma}[Grouping Lemma]
  \label{l:grouping}
  Let $k \geq 0$ and $p \in [0, 1]$. Then
  \begin{equation}
    \label{e:grouping}
    u_{k + 1}(p) \leq \scalezero^6 (\max\{u_k(p), v_k(p)\})^2.
  \end{equation}
\end{lemma}

\begin{proof}
We start by fixing $m \in M_{k+1}$ and a sequence $\xi^x$ for which $I_{m}^x$ is a good interval.
Let $S$ and $\xi^y$ be such that $S$ is a $(k+1)$-fractal.
By the definition of $(k+1)$-fractal sets, $S$ can be written as the union $S=\cup_{i=1}^{2^{\expfractal}} S_i$ of a $k$-ordered family of $k$-fractal sets $S_i$.
We fix $S_1$, $S_2$ and $S_3$ as in Remark~\ref{r:recovery} and the result follows from \eqref{e:recovery} by noting that $\mathcal{R}_{k + 1}(S_i, I)$ is necessarily contained in a single interval of scale $k + 1$.
\end{proof}

Let us now move to the recursion inequality for $v_{k+1}$.

\begin{lemma}[Traversal Lemma]
  \label{l:battle}
  Let $k \geq 0$ and $p \in [0, 1]$. Then
  \begin{equation}
    \label{e:battle}
    v_{k + 1}(p) \leq \scalezero^6 \big(\max\{u_k(p), v_k(p)\}\big)^2.
  \end{equation}
\end{lemma}

\begin{figure}[h]
  \begin{subfigure}{.3\textwidth}
    \centering
    \includegraphics{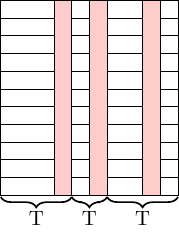}
    \caption{Case 1}
    \label{f:v_k_1}
  \end{subfigure}
  \begin{subfigure}{.3\textwidth}
    \centering
    \includegraphics{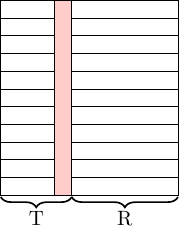}
    \caption{Case 2.1}
    \label{f:v_k_2}
  \end{subfigure}
  \begin{subfigure}{.3\textwidth}
    \centering
    \includegraphics{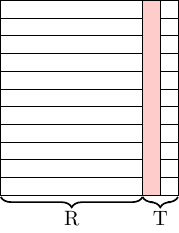}
    \caption{Case 2.2}
    \label{f:v_k_3}
  \end{subfigure}
  \caption{The three cases appearing in the proof of Lemma~\ref{l:battle} corresponding to: \eqref{f:v_k_1} there exist more than one defect that have to be traversed, \eqref{f:v_k_2} there exists a single defect that lies on the left half of the box and \eqref{f:v_k_3} there exists a single defect that lies on the right half of the box..
    The labels T and R, stand respectively to ``traversal'' and ``recovery'';
    Lemma~\ref{l:algebra} is used when the label is $T$ while Lemma~\ref{l:recovery} is used when the label is $R$.
  }
\end{figure}

\begin{proof}
  We may assume that
  \begin{equation}
    \label{e:priori_bound}
    \max\{u_k(p), v_k(p)\} < \scalezero^{-3}
  \end{equation}
  for if it was not case, then the result would hold trivially.
  Fix an index $m \in M_{k+1}$ and a sequence $\xi^x$ such that $h:= H_m^1 \geq 1$.
  Let $S$ and $\xi^y$ be such that $S$ is a $(k+1)$-ordered family of $2^{\expdefect (h-1)}$, $(k+1)$-fractal sets.

  We divide the proof in two cases, depending on whether $B^x_m = k + 1$, or $B^x_m \leq k$.
  In what follows, we assume for simplicity that $m= (k + 1,0)$.

  \medskip

  \noindent {\bf Case 1, ($B^x_m = k + 1$):}
  In this case, $I^x_m$ contains at least two bad sub-intervals.
  Let us denote $r \geq 2$ the amount of such bad sub-intervals.
  Then, there exist indices $0\leq i_1< \ldots< i_r\leq \scalezero-1$ such that $h_j:=H^x_{(k,i_j)} \geq 1$ for $j \in \{1,\ldots, r\}$.
  Moreover, $H^x_{(k,i)}=0$ for every $i\in[0,\scalezero-1]\setminus\{i_1, \cdots, i_r\}$.
  Set $i_0=0$ and define the $k$-blocks:
  \begin{equation}
    C_j=\bigcup_{i=i_{j-1}+1}^{i_j} I^x_{(k,i)}, \textrm{ for }j=1,\cdots, r-1 \,\,\,\, \textrm{ and } \,\,\,\, C_r=\bigcup_{i=i_{r-1}+1}^{L-1} I^x_{(k,i)}
  \end{equation}
  Note that $C_j$ contains the $j$-th bad subinterval $I_{(k,i_j)}$ preceded by the good subintervals between $I_{(k, i_{j-1})}$ and $I_{(k,i_j)}$ (if any). The last one also contains the possible good $k$-intervals to the right of rightmost bad $k$-interval, $I_{(k,i_j)}$ .

  Because $S$ is a $(k+1)$-ordered family containing $2^{\expdefect(h-1)}$ $(k+1)$-fractal sets, it may also be regarded as a $k$-ordered family of $2^{ \expfractal + \expdefect ( h - 1 )}$ $k$-fractal sets.
  Recalling from \eqref{e:H_m} that $h = 1 + \sum_{j=1}^{r} h_j$ we can rewrite this number as
  \begin{equation}
    2^{ \expfractal + \expdefect (h-1) } = 2^{ \expfractal + \expdefect \sum_{j=1}^{r} h_j} .
    \label{e:importance_2}
  \end{equation}
  Informally speaking, by Lemma \ref{l:algebra}, we typically subtract $\expdefect h_j$ to the exponent giving the size of the set $S$ when we cross the vertical strip that projects to the $j$-th $k$-block $C_j$.
  This amounts to a total loss of $(\expdefect \sum_{j=1}^r h_j)$ in the exponent so that we are left with a reminder containing at least $2^{\expfractal}$ $k$-fractal sets, which will form a $k$-ordered family.
  In other words, the remainder contains a $(k+1)$-fractal set.
  We now turn this reasoning into a rigorous argument.

  Define $S_0 = S$ and, inductively, $S_j = \mathcal{R}_k(S_{j - 1}, C_j)$, for $j = 1, \dots, r$.
  Therefore,
  \begin{equation}
    \label{e:v_k+1_case0}
    \begin{split}
      \mathbb{P}^{\xi^x, \xi^y}_p
      & \Big(\bigcap_{j = 1}^{r } \Big\{ \text{$S_{j + 1}$ contains a $k$-ordered family of $2^{\expfractal + \sum_{l = j + 1}^{r}\expdefect h_{i_l}}$ $k$-fractal sets}\Big\} \Big) \\
      & \overset{\eqref{e:lemma_algebra}}{\geq}
        \prod_{j = 1}^{r } \Big( 1 - \scalezero^5 \big( \max\{u_k(p), v_k(p)\} \big)^2  \Big) \overset{\text{Bernoulli}}\geq 1 - \scalezero^6 \big( \max\{u_k(p), v_k(p)\} \big)^2.
    \end{split}
  \end{equation}

  On the event appearing on the left-hand side in \eqref{e:v_k+1_case0}, we have that $\mathcal{R}_k(S,I^x_m)$ contains a $(k+1)$-fractal set.
  Hence,
  \begin{equation}
    v_{k+1}(p) \leq \scalezero^6 \big( \max\{u_k(p), v_k(p)\} \big)^2,
  \end{equation}
  as desired.
  \medskip

  \noindent {\bf Case 2, ($B^x_m < k + 1$):} In this situation, the defect is composed of a single bad column in the previous scale.
  This case is then further split in two possibilities, depending on whether this bad subinterval lies to the left of the middle point of $I_m^x$ or it lies to the right of the middle point of $I_m^x$, see Figure \ref{f:v_k_1}.

  More precisely, let $m' = (k,t) \in \mathcal{Q}_m$, be the single bad subinterval of $I_m^x$.
  Note that $H_{m'}^x = h + 1 \geq 2$.

  {\bf Case 2.1, ``The defect lies on the left half'' $(t < L/2)$:}
  In this case we will first cross the defect and then use the right side to perform a recovery.
  More precisely, since $S$ is a $(k + 1)$-ordered family of $2^{\expdefect (h - 1)}$, $(k + 1)$-fractal sets, we can find $4$ disjoint pieces $S=\cup_{i=1}^{4} S_i$ where each $S_i$ is a $k$-ordered family of $2^{\expdefect h }$ $k$-fractal sets.
  We have by Lemma \ref{l:algebra} with $l=4h$ that
  \begin{equation}
    \begin{split}
      \mathbb{P}^{ \xi^x, \xi^y }_p \bigg(
      \textrm{$\mathcal{R}_k \big( S_i, \cup_{j = 0}^{t} I^x_{ (k, j) } \big)$ contains a $k$-fractal set}
      \bigg)
      \geq 1 - \scalezero^5\big(\max\{u_k(p), v_k(p) \}\big)^2,
    \end{split}
  \end{equation}
  for each $i = 1, \cdots, 4$.
  We now find a lower bound for the event
  \begin{equation}
    A = \Big\{ \displaystyle{
      \textrm{$\mathcal{R}_k \big( S, \cup_{j = 0}^{t} I^x_{ (k, j) } \big)$ contains a $k$-ordered family of three $k$-fractal sets}
    }\Big\}.
  \end{equation}
  More precisely, using a union bound
  \begin{equation}
    \begin{split}
      & \mathbb{P}^{ \xi^x, \xi^y }_p (A)
        \geq \mathbb{P}^{ \xi^x, \xi^y }_p
        \bigg(
        \bigcap_{i=1}^{3}
        \bigg\{
        \begin{split}
          \textrm{$\mathcal{R}_k \big( S_i, \cup_{j = 0}^{t} I^x_{ (k, j) } \big)$ contains}\\ \textrm{ a $k$-fractal set \,\,\,\,\,\,\,\,\,\,\,\,\,\,\,}
        \end{split}
        \bigg\}
        \bigg)
        \geq 1 - 3\cdot \scalezero^5\big(\max\{u_k(p), v_k(p) \}\big)^2.
    \end{split}
    \label{e:ate_o_defeito}
  \end{equation}
  Having crossed the defect, we let $I = \cup_{i = t + 1}^{\scalezero-1} I^x_{( k, i)}$, which is a good $k$-block.
  Conditioning on $A$ and on the realization of the $k$-ordered pair $S'_1$,  $S'_2$ and $S'_3$ of $k$-fractals in the remainder of $S_1, S_2$ and $S_3$, we can use Lemma \ref{l:recovery} (see also \eqref{e:recovery}) to estimate
  \begin{equation}
    \label{e:recovery_2}
    \begin{split}
      & \mathbb{P}^{\xi^x, \xi^y}_p  \big(\text{$\mathcal{R}_{k+1}(S, I_m^x)$ does not contain a $(k+1)$-fractal set}\big |\, A, S'_1, S'_2, S'_3) \\
      & \leq \mathbb{P}^{\xi^x,\xi^y}_p \bigg( \bigcap_{i=1,2}  \big\{ \mathcal{R}_{k+1}(S'_i, I) \textrm{ does not contain a $(k+1)$-fractal in $Z_{k+1}(S'_i)$} \big \} \bigg{|} \, A, S'_1, S'_2, S'_3 \bigg)\\
      & \overset{\eqref{e:recovery}}{\leq} \scalezero^5 \big(\max\{u_k(p),v_k(p)\}\big)^2.
    \end{split}
  \end{equation}
  From \eqref{e:ate_o_defeito} and \eqref{e:recovery_2} we conclude that
  \begin{equation}
    \begin{split}
      v_{k+1}(p)  \leq 4\cdot \scalezero^5(\max\{u_k(p), v_k(p) \})^2  \leq \scalezero^6(\max\{u_k(p), v_k(p) \})^2,
    \end{split}
  \end{equation}
  as claimed.

  {\bf Case 2.2, ``The defect lies on the right half'' $(t \geq L/2)$:}
  In this case we will perform a recovery on the set $S$ before crossing the defect.
  More precisely, recall that $S$ a $(k+1)$-ordered family of $2^{\expdefect (h - 1)}$ $(k + 1)$-fractal sets, say $S = S_1 \cup \cdots \cup S_{2^{\expdefect (h - 1)}}$.
  For each $(k+1)$ fractal set $S_i$, we may write $S_i = \cup_{j = 1}^{2^\expfractal} \widetilde{S}_{i, j}$ where the union runs over a $k$-ordered collection of $k$-fractal sets ${S}_{i,j}$.
  In particular ${S_{i}}$ contains three $k$-fractal sets $S_{i,1}$, $S_{i,2}$ and $S_{i,3}$ that lie inside good $(k+1)$-intervals.
  Therefore, by Lemma \ref{l:recovery}, we have for each $i = 1,\ldots, 2^{4(h-1)}$,
  \begin{eqnarray}
    \label{e:recuperacao_em_cada_caixa_k+1}
    \mathbb{P}^{ \xi^x, \xi^y }_p
    \Bigg(  \displaystyle{
    \begin{split}
      \mathcal{R}_k \big( {S_i}, \cup_{j = 0}^{t - 1} I^x_{ (k, j) } \big) \text{ contains a} \,\,\,\,\,\,\,\,\,\, \\
      \text{ $k$-ordered family of $(\scalezero - 1)$ $k$-fractal sets}
    \end{split}}
    \Bigg)
    \geq 1- \scalezero^5 \max \{ u_k(p), v_k(p) \} )^2.
  \end{eqnarray}
  Let us define $\alpha =  1 - \scalezero^{5} (\max\{u_k(p),v_k(p)\})^2 \geq 1 - \scalezero^{-1}$.
  We may assume $\alpha \geq 0.9$, because otherwise \eqref{e:battle} would immediately hold.

  We now apply the concentration bound \eqref{e:concentration} to the random variables
  \begin{equation}
    X_i = \mathbbm{1} \big\{\mathcal{R}_k \big( {S_i}, \cup_{j = 0}^{t - 1} I^x_{ (k, j) } \big) \text{ contains a $k$-ordered family of $L-1$ $k$-fractal sets}\big\},
  \end{equation}
  If at least half of these random variables attain value $1$ then we will be able to reach as many as \begin{eqnarray}\label{e:L_is_very_large}
  (L-1) 2^{\expdefect(h - 1) } / 2 \geq 2^{ \expfractal + \expdefect h } \textrm{ \ $k$-fractal sets.}
  \end{eqnarray}
To do this, we use that $L=10^6\geq 2^{15}+1$. This means
  \begin{eqnarray}
    \mathbb{P}^{ \xi^x, \xi^y }_p
    \Bigg(  \displaystyle{
    \begin{split}
      \mathcal{R}_k \big( S, \cup_{j = 0}^{t - 1} I^x_{ (k, j) } \big) \text{ does not contain a } \\
      \text{ $k$-ordered family of $2^{ \expfractal + \expdefect h }$ $k$-fractal sets}
    \end{split}}
    \Bigg)
    \leq 10 \cdot \scalezero^5 ( \max\{ u_k(p), v_k(p) \} )^2, \label{e:traversal_before_defect}
  \end{eqnarray}

  If the event displayed in \eqref{e:traversal_before_defect} occurs, we may condition on the realization of a $k$-ordered family of $2^{ \expfractal +\expdefect h}$ $k$-fractal sets $S'\subseteq \mathcal{R}_k \big( S, \cup_{j = 0}^{t - 1} I^1_{ (k, j) } \big)$.
  Applying Lemma \ref{l:algebra} to $S'$, we conclude that
  \begin{eqnarray}\label{e:travessia_no_defeito}
    \mathbb{P}^{ \xi^x, \xi^y }_p
    \Bigg(  \displaystyle{
    \begin{split}
      \mathcal{R}_k \big( S', \cup _{i=t}^{\scalezero -1}I^x_{ (k, i) } \big) \text{ does not contain a\,\,\,\,} \\
      \text{ $k$-ordered family of $2^{ \expfractal }$ $k$-fractal sets}
    \end{split}}
    \Bigg)
    \leq \scalezero^5 ( \max \{ u_k(p), v_k(p) \} )^2.
  \end{eqnarray}
  Note that a $k$-ordered family of $2^{\expfractal}$ $k$-fractal sets form a $(k+1)$-factral, so  the result follows by \eqref{e:traversal_before_defect} and \eqref{e:travessia_no_defeito}, finishing the proof in all cases.
\end{proof}

\subsection{Proof of Lemma~\ref{l:summability}}
\label{s:summabilitry}

We are now in position to proof the decay of $v_k$ and $u_k$.

\begin{proof}[Proof of Lemma \ref{l:summability}]
We proceed by induction on $k$.
Applying Lemmas \ref{l:grouping} and \ref{l:battle} we obtain that, for every $k\geq 0$:
\begin{eqnarray*}
  \max\{u_k(p), v_k(p)\} \leq \scalezero^{ - k - 10} \ \ \ \ \text{implies} \ \ \ \ \max\{u_{k + 1}(p), v_{k + 1}(p)\} \leq \scalezero^{ - (k + 1) - 10}.
\end{eqnarray*}
Hence it suffices to show that $\max\{u_0(p), v_0(p)\} \leq L^{-10}$ whenever $p \geq (1-L^{-10})^{1/4}$.

Recalling that at scale $k=0$ intervals are simply points in $\mathbb{Z}$ and that $u_0(p) = 1 - p$, we obtain:
\begin{equation}
\text{$u_0(p) \leq \scalezero^{ - 10}$ \ \ \ \ as soon as \ \ \ \ $p \geq (1 - \scalezero^{ - 10})^{1/4} \geq (1 - \scalezero^{ - 10})$.}
\end{equation}

Now it only remains to show that $v_0(p) \leq \scalezero^{ - 10}$ whenever $p > (1 - \scalezero^{ -10 })^{1/4}$.
For this, fix $m=(0,0)$, and assume that $\xi_0^x=h$ for some $h\geq 1$.
Fix $\xi^y$ and let $S$ be a $0$-ordered family of $2^{\expdefect (h - 1)}$ $0$-fractal sets for $\xi^y$.
We will consider two cases depending on how large $h$ is.
\medskip

\noindent {\bf Case 1.} $1\leq h\leq 3$:
Let $x$ be any point in $S$ (e.g.\ the smallest point)
Then
\begin{equation}
\label{e:bound_v0_1}
\begin{split}
 \mathbb{P}^{\xi^x, \xi^y}_p & \big(\mathcal{R} (S,I^1_{(0,0)}) \text{ does not contain a $0$-fractal set}\big ) \\
 & \leq \mathbb{P}^{\xi^x,\xi^y}_p \big(\textrm{the edge between $(0,x)$ and $(1,x)$ is closed}\big) \\
& = 1- p^{h+1} \leq 1- p^4 \leq  \scalezero^{-10},
 \end{split}
 \end{equation}
where the last inequality holds because $p \geq (1 - \scalezero^{ -10 })^{1/4}.$

\medskip

\noindent {\bf Case 2.} $h\geq 4$.
We have that
\begin{equation}
\label{e:bound_v0_2}
\begin{split}
  \mathbb{P}^{\xi^x, \xi^y}_p
  & \big(\mathcal{R} (S,I^1_{(0,0)}) \text{does not contain a $0$-fractal set}\big )\\
  & \leq \mathbb{P}^{\xi^x,\xi^y}_p \Big( \bigcap_{ x \in S}\{\textrm{the edge between $(0,x)$ and $(1, 0)$ is open}\} \Big) \\
  & =\big(1 - p^{ h + 1}\big)^{2^{\expdefect (h - 1)}} \leq \exp\big( - 2^{\expdefect (h - 1) }p^{ h + 1 }\big) \\
  & \overset{(h\geq 4)}\leq \exp \big(-(16p)^{h  + 1}\big)  \overset{(h\geq 4,\, p> 1/4)}\leq \exp \big(-(16p)^{5}\big) \leq \scalezero^{ -10 }.
\end{split}
\end{equation}
To see why the last inequality holds, note that $p \geq \tfrac{1}{16}(10 \ln(\scalezero))^{1/5}\approx 0.167$ (because $p \geq 1-L^{-10}$ and $L=10^6$).

Using \eqref{e:bound_v0_1} and \eqref{e:bound_v0_2} and maximizing over the choices of $\xi^x$, $\xi^y$ and $S$ we obtain that $v_0(p) \leq L^{-10}$ which concludes the proof.
\end{proof}

\section{Proof of Theorem~\ref{t:main}}
\label{s:summability}

Having established the decay of $u_k$ and $v_k$ in Lemma~\ref{l:summability}, we are now in position to prove Theorem~\ref{t:main}, which will follow easily once we establish Lemma~\ref{l:build_cluster} below.

\begin{lemma}
  \label{l:build_cluster}
  Let $\xi^x$ and $\xi^y$ be such that,
  \begin{equation}
    \label{e:only_good_intervals}
    \text{for every $ k \geq 0$, the intervals $ I_{ (k, 0) }^x$ and $I_{(k,0)}^y$ are good}.
  \end{equation}
  If $p$ is such that $\max\{ u_k(p), v_k(p) \}$ forms a summable sequence in $k$, then
  \begin{equation}
    \mathbb{P}^{\xi^x,\xi^y}_p\big(\text{$o$ is connected to infinity}\big)>0.
  \end{equation}
\end{lemma}
The proof of Lemma \ref{l:build_cluster} follows from a Borel-Cantelli argument.

\begin{proof}[Proof of Lemma \ref{l:build_cluster}]
Since we are assuming that the sequence is summable we can fix a positive integer $k_0$ such that $\sum_{k = k_0} \max \{ u_k(p), v_k(p) \} < (4L)^{-1}$.

Fix a realization of the sequences $\xi^x, \xi^y$ for which \eqref{e:only_good_intervals} holds.
We have that $o$ belongs to an infinite connected component as soon as
\begin{display}
  all the edges along $\{0\}\times I^y_{( k_0, 0)}$ are open, and for every $k \geq k_0$,\\
  both the corridor $I^x_{ (k + 1, 0) }\times I^y_{ (k, 0) }$ is traversed horizontally\\
  and the corridor  $I^x_{ (k, 0) } \times I^y_{ (k + 1, 0)}$ is traversed vertically.
\end{display}
Therefore, we can use Lemma \ref{l:corridor} and the FKG inequality to obtain
\begin{equation}
  \begin{split}
    \mathbb{P}^{ \xi^x, \xi^y }_p  \big( o \leftrightarrow \infty \big)
    &\geq p^{L_{k_0}}\prod_{k = k_0}^{\infty} \big( 1 - \scalezero\max\{u_k(p), v_k(p)\}\big)^2 \\
    & \geq p^{L_{k_0}} \Big(1 - 2L \sum_{k = k_0}^{\infty} \max \{ u_k(p), v_k(p) \} \Big) \geq \frac{p^{L_{k_0}}}{2} > 0,
  \end{split}
\end{equation}
finishing the proof of the lemma.
\end{proof}

We finish the present section showing how the proof of our main result follows from Lemmas~\ref{l:build_cluster}, \ref{l:environment} and \ref{l:summability}.
\begin{proof}[Proof of Theorem \ref{t:main}]
Since $\rho \leq L^{-16}$, Lemma \ref{l:environment} together with a standard Borel-Cantelli argument allows us to conclude that the random environment $\xi^x, \xi^y$ is such that
\begin{equation}
\label{e:only_good_intervals_2}
\text{for every $ k \geq 0$, the intervals $ I_{ (k, 0) }^x$ and $I_{(k,0)}^y$ are good}
\end{equation}
with positive probability.
Fix such $\xi^x$ and $\xi^y$.

Since $p > (1-L^{-10})^{1/4}$, Lemma \ref{l:summability} gives that \eqref{e:summability} holds, from where we conclude that $\max\{u_k(p), v_k(p)\}$ is summable.
Hence we can apply Lemma~\ref{l:build_cluster} in order to conclude that there exists an infinite cluster $P^{\xi^x, \xi^y}_p$-a.s.
Hence, $\mathbb{P}\big(\text{$o$ is connected to infinity}\big) > 0$.
\end{proof}

\section{Absence of phase transition}
\label{s:sharpness}

Up until now we have assumed that the random variables in both sequences $\xi^x$ and $\xi^y$ are geometrically distributed as in \eqref{e:distr_xi}, hence they have exponential tail decay.
The main goal in this section is to prove Theorem \ref{t:sharpness} which states that this tail assumption cannot be dropped.

\begin{proof}[Proof of Theorem \ref{t:sharpness}]
The first step in the proof is to choose a large width for a column that would make it very hard to be crossed.
More precisely, for every $p < 1$, we pick $k = k(p) \in \mathbb{Z}_+$ such that $p + p^k < 1$.
The intuition behind this choice is that: on the event that $\xi^x_j = k$ for every $j$, even if we had $\xi^y_j = 1$ for every $j$, the resulting independent percolation process would still be subcritical (see \cite[Theorem 11.115]{Gri99}, page 332).
By standard results on the exponential decay of connectivity on the subcritical phase (see e.g.\ \cite{duminil2016new}) we obtain a constant $\gamma = \gamma(p) > 0$ such that
  \begin{eqnarray}
    \label{e:exp_decay_anisotropy_lattice}
    \mathbb{P}_p \big( o \leftrightarrow \partial B(o, m) \text{ on the lattice $(k \mathbb{Z}) \times \mathbb{Z}$}\big)\leq e^{- \gamma m},
  \end{eqnarray}
  for every $m \geq 1$.

Since we are assuming in \ref{l:no_transition:1}) that $\xi^x_0$ does not have compact support, if we denote $\alpha = \alpha(p) :=\mathbb{P}(\xi_0^x \geq k)$, then we have $\alpha > 0$.
Denote also $\beta = \beta(p):=\frac{1}{2\log(1/\alpha)}$ and fix $\delta<\gamma \beta$.

Back to the random lattice, we consider the nested increasing sequence of rectangles
  \[
    R_n=[-e^n,e^n]\times[-e^{\delta n}, e^{\delta n}].
  \]
  Write $l_1(n)$, $l_2(n)$, $l_3(n)$ and $l_4(n)$ for the right, top, left and bottom sides of $R_n$, respectively.
  Note that for every environment $\xi^x$ and $\xi^y$
  \begin{eqnarray}\label{e:cross_decomposition}
    \mathbb{P}^{\xi^x,\xi^y}_p(o\leftrightarrow\infty)\leq \liminf_{n\to\infty}\sum_{t=1}^4\mathbb{P}^{\xi^x,\xi^y}_p\big(o\leftrightarrow l_t(n) \textrm{ in }R_n\big).
  \end{eqnarray}

We now look within our random environment for a long stretch for which $\xi_j^x \geq k$.
Denoting
\begin{equation}
A_n:=\big\{\textrm{there exists $i \in [1, e^n)$ such that $\xi_j^x \geq k$ for every $j \in [i, i + \beta n)$} \big\},
\end{equation}
we have
\begin{equation}
  \begin{split}
    \mathbb{P}(A_n)
    & \geq \mathbb{P} \Big( \bigcup_{q = 1}^{\lfloor e^n/ \lceil \beta n \rceil \rfloor}\bigcap_{j = (q - 1) \lceil \beta n \rceil}^{ q \lceil \beta n \rceil - 1} \big\{ \xi_j^x \geq k \big\} \Big)
      =  \Big[ 1 - ( 1 - \alpha^{\lceil \beta n \rceil} )^{ \lfloor e^n / \lceil \beta n \rceil \rfloor} \Big]\\
    &  \geq  \Big[ 1 - \exp \big\{ -\alpha^{\beta n + 1} \big( \tfrac{e^n}{\beta n} - 1 \big) \big\} \Big]
      \geq \Big[ 1 - \exp \Big\{ - \; \frac{\alpha \exp(\tfrac{n}{2})}{\beta n} + 1 \Big\} \Big] \xrightarrow[n\to\infty] {}1.
    \end{split}
    \label{e:large_defect_horizontal}
  \end{equation}
On the event displayed in left-hand side on \eqref{e:large_defect_horizontal}, we obtain columns that are very difficult to cross, regardless of the realization of $\xi^y$.
  More precisely, on the event that $\xi_j^x\geq k$ for all $j\in[ i, i + \beta n)$,  \eqref{e:exp_decay_anisotropy_lattice} yields that for every $z \in \{ i \} \times \mathbb{Z}$,
  \begin{eqnarray}
    \label{e:exp_decay_anisotropy_column}
    \mathbb{P}^{\xi^x,\xi^y}_p\big(z \leftrightarrow \{i + \beta n\}\times\mathbb{Z} \textrm{ in }[ i, i + \beta n) \times \mathbb{Z}\big)\leq e^{- \gamma \beta n}.
  \end{eqnarray}

By \eqref{e:exp_decay_anisotropy_column}, on the event  $A_n$ for almost every environment $\xi^x$, $\xi^y$, one has
  \begin{eqnarray}\label{e:prob_zero_crossing_horizontal}
   \mathbb{P}^{\xi^x,\xi^y}_p(o\leftrightarrow l_1(n) \textrm{ in }R_n)+\mathbb{P}^{\xi^x,\xi^y}_p(o\leftrightarrow l_3(n) \textrm{ in }R_n)
    \leq 2\cdot(2 e ^{\delta n}+1)e^{-\gamma\beta n} \xrightarrow[n\to\infty]{}0.
  \end{eqnarray}

In order to bound the other two terms appearing in the sum in \eqref{e:cross_decomposition} we will regard the vertical environment $\xi^y$ and show that there is a subsequence of rectangles $R_{n}$ containing large enough vertical defects.

For that end, let us fix $a > - \big(\log p\big)^{-1}$ and define
\begin{equation}
  B_n= \big\{ \max_{1 \leq i \leq e^{\delta n} }\{ \xi_i^y \}> an \big\}.
\end{equation}
We have
  \begin{equation}
    \label{e:limsup_Bn}
    \begin{split}
      \limsup_n \mathbb{P}\big(B_n\big)
      & = \limsup_n \Big[  1 - \big( 1 - \mathbb{P}( \xi_0^y >  an ) \big)^{ \lfloor e^{\delta n} \rfloor }\Big] \\
      & \geq \limsup_n \Big[ 1- \exp\big(-e^{\delta n} \mathbb{P}( \xi_0^y >  an ) \big) \Big] = 1,
    \end{split}
  \end{equation}
  where the above equality follows from the second Hypothesis~\ref{l:no_transition:2} in the statement of theorem.

  Set $j_{\max}:= \text{argmax}\{\xi^y_i \colon 1\leq j \leq e^{\delta n}\}$.
 On the event $B_n$, for almost every environment $\xi^x$, $\xi^y$, we have that
  \begin{align}
    \label{e:prob_zero_crossing_vertical}
    & \; \mathbb{P}^{\xi^x,\xi^y}_p(o\leftrightarrow l_2(n) \textrm{ in }R_n)+\mathbb{P}^{\xi^x,\xi^y}_p(o\leftrightarrow l_4(n) \textrm{ in }R_n) \nonumber \\
                          & \leq 2\, \mathbb{P}^{\xi^x,\xi^y}_p\bigg(\bigcup_{-e^{-n}\leq i \leq e^n}\{\textrm{the edge between $(i ,j_{\max})$ and $(i,j_{\max}+1)$ is open}\}\bigg) \nonumber \\
                          & \leq 2 \cdot ( 2 e^{ n } + 1) p^{ an } \xrightarrow[n\to\infty]{} 0.
  \end{align}
Where the limit vanishes thanks to the definition of $a$.
Plugging into \eqref{e:cross_decomposition} the estimates \eqref{e:prob_zero_crossing_horizontal}, \eqref{e:limsup_Bn} and \eqref{e:prob_zero_crossing_vertical}, we obtain the desired result.
\end{proof}

\section{Oriented percolation in a random environment}
\label{s:ksv}

In this section we study two other models, in order to illustrate how our techniques can be adapted to tackle different problems.
These models can be regarded as oriented versions of the one considered in Theorem~\ref{t:main} and one of them has already been studied in \cite{Kesten22} and we refer to it from now on as the KSV Model.

The oriented square lattice is the graph whose vertex set is given by
\begin{equation}
  \mathbb{V} = \big\{ (i, j) \in \mathbb{Z}^2;\, i + j \textrm{ is even} \big\}
\end{equation}
and the oriented edges are
\begin{equation}
  \mathbb{E} = \Big\{ \big( (i_1, j_1), (i_2, j_2) \big) \in \mathbb{V}^2 ;
  \, i_2 = i_1 + 1 \textrm{ and } |j_2 - j_1| = 1 \Big\}.
\end{equation}
Note that this graph is isomorphic to the usual square lattice with edges oriented to the right and upwards (the isometry is given by a rotation of $\pi/4$, followed by a dilation of $\sqrt{2}$).

We now introduce the two models that we are able to deal with our techniques.

\paragraph{KSV}

We start be re-introducing a process studied in \cite{Kesten22}.
Let $(\eta_i)_{i \in \mathbb{Z}}$ be a sequence of i.i.d. Bernoulli random variables with density $\rho \in (0, 1)$.
Fixed this sequence, we now chose two parameters $p_g, p_b \in (0, 1)$ (respectively representing good an bad columns) and we independently declare each vertex $(i, j) \in \mathbb{V}$ open with the following probabilities:
\begin{equation}
  \label{e:p_good_p_bad}
  \mathbb{P} \big[ (i, j) \text{ is open} \big] =
  \begin{cases}
    p_g, \qquad \text{if $\eta_i = 0$,}\\
    p_b, \qquad \text{if $\eta_i = 1$.}\\
  \end{cases}
\end{equation}

The main result of \cite{Kesten22} is given by the following
\begin{theorem}[KSV]
  \label{t:ksv}
  For any $p_g > p_c$ ($p_c$ stands for oriented, site percolation) and any $p_b > 0$, there exists $\bar{\rho} = \bar{\rho}(p_g, p_b) > 0$ such that
  \begin{equation}
    \label{e:ksv}
    \mathbb{P}_{p_g, p_b}^\eta \big[ 0 \to \infty \big] > 0,
  \end{equation}
  for almost every $\eta$ under the $Ber(\bar{\rho})$ i.i.d. distribution.
\end{theorem}

Below we introduce another oriented model that we consider in this section.

\paragraph{Oriented percolation with geometric defects}

Consider a single sequence $\xi=(\xi_i)_{i \in \mathbb{Z}}$  of  i.i.d.\ $\Geo(1-\rho)$ random variables, where $\rho \in [0, 1]$.
Fixed $\xi$ and $p \in [0, 1]$, we define a bond percolation on $\mathbb{V}$, where an edge $\{(i, j), (i + 1, j + 1)\}$ or $\{(i, j), (i + 1, j - 1)\}$  is declared open (independently of others) with probability $p^{\xi_i + 1}$ and closed otherwise.
Denoting by $P^{\xi}_p$ the law of this process, we will prove the following.

\begin{theorem}\label{t:oriented}
Let $L=10^{6}$. As long as $\rho\leq L^{-16}$ and $p\geq (1-L^{-10})^{1/4}$, we have
\begin{equation}P^{\xi}_p[ (0,0) \textrm{ is connected to infinity}]>0.
\end{equation}
for a.e. environment $\xi$.
\end{theorem}

In Subsection~\ref{ss:proof_oriented}, we prove Theorem~\ref{t:oriented}.
But before, let us show that Theorem~\ref{t:ksv} can be obtained from Theorem~\ref{t:oriented} through a simple one-step renormalization argument.

\subsection{Proof of Theorem~\ref{t:ksv} assuming Theorem~\ref{t:oriented}}
\label{ss:oriented_gives_ksv}

In this subsection we assume the validity of Theorem~\ref{t:oriented} (which will be proved in Subsection~\ref{ss:proof_oriented}) in order to provide a proof of Theorem~\ref{t:ksv}.
This proof will follow a one-step renormalization argument and it does not contain any new ideas, but it is included in detail below for the sake of completeness.

\begin{proof}[Proof of Theorem~\ref{t:ksv}]
  Instead of proving the full result, we will only establish that
  \begin{display}
    \label{e:ksv_p_large}
    there exists $p \in (0, 1)$ such that for any $p_b > 0$, there exists $\rho(p_b) > 0$ such that
    $\mathbb{P}_{p, p_b}^\eta \big[ 0 \to \infty \big] > 0,$
    for almost every $\eta$ under the $Ber(\bar{\rho})$ i.i.d. distribution.
  \end{display}
  The fact that the above implies the main theorem has already been proved in Section~8 of \cite{Kesten22} and involves yet another one-step renormalization argument.

  We can make another reduction step, by observing that Theorem~\ref{t:oriented} implies the following
  \begin{display}
    \label{e:oriented_M}
    for any $M \geq 1$, there exists $\rho = \rho(M) > 0$ such that for $p = (1 - L^{-10})^{1/4}$ we have $P_p^{M \xi}[(0, 0) \text{ is connected to infinity}] > 0$,
  \end{display}
  where one can see by the super-script in $P$, that the geometric random variables $\xi$ have been multiplied by $M$.
  To see why the above is true, it is enough to observe that $M$ times a geometric random variable is stochastically dominated by another geometric random variable with smaller parameter.
  Our objective now is to prove that \eqref{e:oriented_M} implies \eqref{e:ksv_p_large}.
  But before that we state two classical results.

  Throughout our argument we will arrive to a $10$-dependent measure on $\{0, 1\}^\mathbb{Z}$.
  To deal with this dependence, we will use a classical result from Liggett, Schonmann, and Stacey.
  More precisely, the main result of \cite{10.1214/aop/1024404279} implies that there exists a $q^*$, such that:
  \begin{display}
    \label{e:10-dep}
    any $10$-dependent measure on $\{0, 1\}^\mathbb{Z}$ whose marginal probabilities are at least $q^*$ stochastically dominates an i.i.d. measure with $p = (1 - L^{-10})^{1/4}$ as in \eqref{e:oriented_M}.
  \end{display}

  In our one-step renormalization, we will introduce a connection event in a $5\ell$ by $4\ell$ box, which plays the role of an edge crossing.
  Using some classical arguments of oriented percolation (in our case it suffices to use the statement of Lemma~6.1 in \cite{NTT19}) we conclude that there exists $p \in (0, 1)$ (called $\gamma$ in \cite{NTT19}) close enough to one and $\ell$ large enough such that
  \begin{equation}
    \label{e:dominate_oriented}
    \inf_{\substack{S \subseteq \{0\} \times [0, \ell)\\|S| \geq \ell/10}}
    \mathbb{P}_p \Big[ S \text{ is connected within $[0, 5\ell] \times [-\ell, 3 \ell)$ to $\ell/10$ sites in $\{5 \ell\} \times [\ell, 2 \ell)$}  \Big] \geq q^*,
  \end{equation}
  where $e = \big( (0, 0), (1, 1) \big)$.

  We are now in position to chose the constants in \eqref{e:oriented_M} and \eqref{e:ksv_p_large} for the rest of the proof to work.
  Take $M \geq 1$ such that
  \begin{equation}
    \label{e:tube_open}
    (1 - L^{-10})^{M/4} \leq p_b^{\ell^2}
  \end{equation}
  and $\bar{\rho}$ (see \eqref{e:ksv_p_large}) such that
  \begin{equation}
    \label{e:prob_bad}
    \Big( 1 - (1 - \bar{\rho})^{5 \ell} \Big) \leq \rho(M).
  \end{equation}
  We are finally in position to give the one-step renormalization argument.

  Before we present the proof, let us give an informal overview of its main steps.
  We start by introducing, for each site $(i, j) \in \mathbb{V}$, the interval $I_{(i, j)} = \{5 i \ell\} \times [j \ell, (j + 1) \ell)$.
  These will play the role of renormalized sites.

  Analogously, we associate for every edge $\big( (i, j), (i + 1, j') \big) \in \mathbb{E}$, the box
  \begin{equation}
    \label{e:B_w}
    B_w := \mathbb{V} \cap \Big[5 i \ell, 5 (i + 1) \ell \Big) \times \Big[ \big(\min\{j, j'\} - 1 \big) \ell , \big( \max \{j, j'\} + 2 \big) \ell \Big),
  \end{equation}
  which is isomorphic to the box appearing in \eqref{e:dominate_oriented}.
  The proof will follow an exploration of $\mathbb{V}$ from left to right.
  In other words, we will proceed over $\mathbb{V}$ through the layers $\mathbb{V}_i = (\{i\} \times \mathbb{Z}) \cap \mathbb{V}$ and we also define $\mathbb{E}_i$ as the set of edges in $\mathbb{E}$ connecting $\mathbb{V}_i$ to $\mathbb{V}_{i + 1}$.

  This exploration will proceed inductively over $i$ defining the set $\mathcal{G}(i) \subseteq \mathbb{V}_i$, which gives the indices of intervals $I_z$, for $z \in \mathbb{V}_i$ that have been reached during our exploration.
  Also, for each such a $z \in \mathcal{G}(i)$ we will specify the ``seed'' $S_z \subseteq I_z$, which corresponds to the points that have been reached in $I_z$.

  More precisely, let $\mathcal{G}(0) = \{0\}$ (``only the origin is accessible in the first layer'') and let $S_{(0, 0)} = I_{(0, 0)}$ (``the whole interval is the seed'').
  Suppose now that for $i \geq 0$ we have defined $\mathcal{G}(i) \subseteq \mathbb{V}_i$ and for every $z \in \mathcal{G}(i)$ we have defined the seed $S_z \subseteq I_z$ with $|S_z| \geq \ell/10$.
  Then we define, for every $(z, w) \in \mathbb{E}_i$ the crossing events
  \begin{equation}
    \label{e:crossing_oriented}
    \mathcal{F}_{(z, w)} = \bigg[
    \begin{gathered}
      \text{there exists at least $\ell/10$ points in $I_w$}\\
      \text{that can be reached from $S_z$ within $B_{(z, w)}$}
    \end{gathered}
    \bigg],
  \end{equation}
  and if the above event occurs, we say that $w \in \mathcal{G}(i + 1)$ and we define $S_w$ as an arbitrary choice of such $\ell/10$ points.

  Our theorem now follows directly if we manage to prove the following claim
  \begin{display}
    \label{e:reduction_oriented}
    the points $(i, j) \in \mathbb{V}$ such that $j \in \mathcal{G}(i)$ dominate the percolation process in \eqref{e:oriented_M}.
  \end{display}
  To prove \eqref{e:reduction_oriented}, we consider two cases:

  {\bf Case 1 -} If $i$ is a good column then the events $(\mathcal{F}_{(z, w)})_{(z, w) \in \mathbb{E}_i}$ are $10$-dependent.
  Moreover, the probability of each such event is at least $q^*$ by \eqref{e:dominate_oriented}.
  Therefore the distribution of the vector $(\mathcal{F}_{(z, w)})_{(z, w) \in \mathbb{E}_i}$ dominates a product of independent Bernoulli variables with parameter $(1 - L^{-10})^{1/4}$, by \eqref{e:10-dep}.
  Note that the above probability corresponds to crossing an edge with $\xi = 0$ in \eqref{e:oriented_M}.

  {\bf Case 2 -} If $i$ is bad, then we lower bound the probability of $\mathcal{F}_{(z, w)}$, for $(z, w) \in \mathbb{E}_i$, by the probability that one point in $S_z$ infects $\ell/10$ points in $I_w$.
  To guarantee independence, we should also make sure that the microscopic edges used to construct these events are disjoint as we vary the pair $(z, w) \in \mathbb{E}_i$.
  This can be done with a simple construction as that of Figure~\ref{f:crossing_bad_ksv}.
  Since one needs to open at most $\ell^2$ edges in the above construction, our lower bound becomes $p_b^{\ell^2}$, which by \eqref{e:tube_open} is at least $(1 - L^{-10})^{M/4}$.
  Note that the latter probability corresponds to a value $\xi = 1$ in \eqref{e:oriented_M}.

  \begin{figure}
    \centering
    \begin{tikzpicture}[scale=.5]
      \draw (0,0) rectangle (20,4);
      \foreach \i in {0,1,...,10}
      \draw (\i * .2, 3.2 - \i * .2)--(\i * .2 + .2 ,3.2 - \i * .2 -.2);
      \foreach \i in {6,7,...,49}
      \draw (\i * .4, 1.2)--(\i * .4 + .2 ,1);
      \foreach \i in {5,6,...,49}
      \draw (\i * .4 +.2, 1)--(\i * .4 + .4 ,1.2);
      \foreach \i in {0,1,2}
      \draw (19.8-\i*.4,1)--(20,0.8-\i*.4);
      \foreach \i in {0,1}
      \draw (19.6-\i*.4,1.2)--(20, 1.6 + \i* .4);
      \draw[blue, ultra thick] (0, 2) -- (0, 4);
      \draw[blue, ultra thick] (20, 2) -- (20, 0);
    \end{tikzpicture}
    \caption{To cross a bad interval, one can build a connection by opening a deterministic set with at most $\ell^2$ edges (at a cost of $p_b^{\ell^2}$).
      For distinct edges $(z, w) \in \mathbb{E}_i$, the edges used in this construction are disjoint, thus the corresponding events are independent.}
    \label{f:crossing_bad_ksv}
  \end{figure}
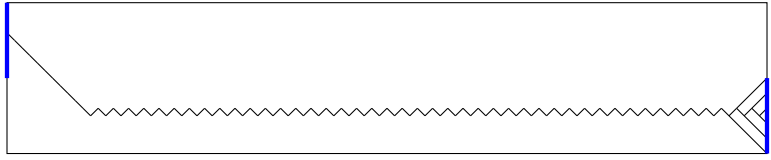

  To finish the proof we recall that the model from \eqref{e:oriented_M} is super-critical.
  Observe also that it still percolates if we truncate the defect variables $\xi$ in such a way that it only takes values $0$ and $1$ (after all, this is a monotone modification which only removes large defects).
  By \eqref{e:prob_bad}, the probability that $\xi = 1$ is at least the probability to find a bad block in the KSV model.
  Finally, the case analysis above shows that the exploration of the KSV model dominates the corresponding exploration for the model in \eqref{e:oriented_M} (both when a column is good and $\xi = 0$ and when the column is bad and $\xi = 1$).
  This shows \eqref{e:reduction_oriented}, finishing the proof of the theorem.
\end{proof}

\subsection{Proof of Theorem~\ref{t:oriented}}
\label{ss:proof_oriented}

Theorem \ref{t:oriented} will be proved in a similar way as Theorem \ref{t:main}.
Indeed the control of the random environment is done exactly in the same manner, meaning that we will use Lemma~\ref{l:environment} as is.
For the control on the percolation process, there are a few differences due to the oriented nature of the model considered in Theorem~\ref{t:oriented}; for example it is no longer possible to find vertical crossings as those obtained in Lemma~\ref{l:branch}.

In order to adapt our proof to the oriented setting, we will start by replacing the square boxes introduced in \eqref{e:boxes} by rectangles that are horizontally. elongated.
For this, we will write as before, $L=10^6$ and consider the following scale progression:
\begin{eqnarray*}
 L_k^x=L^k \qquad \textrm{ and } \qquad L_k^y=(L/10)^k,
\end{eqnarray*}
for all $k\geq 0$.
Observe that $L^y_k = (L^x_k)^{5/6}$ for every $k$.

For $m=(k,i)\in M_k$ let $I^x_m=[ i L_k^x, (i+1) L_k^x)$ and $I^y_m=[ i L_k^y, (i+1) L_k^y)$ denote the horizontal and vertical $k$-intervals respectively.
Notice that the horizontal intervals at scale $k$ are much longer than their vertical counterparts.

Next we define the boxes at scale $k$, which will be indexed by edges of the microscopic lattice.
More precisely, for $k \geq 0$, we define $W_k = \{k\} \times \mathbb{E}$ and for $\bar{e} = (k, e) \in W_k$, one writes $e = \big( (i, j), (i + 1, j') \big)$ with $|j - j'| = 1$ to define the boxes
\begin{eqnarray}
  \label{e:B_bar_e}
  B_{\bar{e}} := \mathbb{V} \cap \Big[i L^x_k , (i + 1) L^x_k  \Big) \times \Big[\min\{j, j'\} L^y_k , \big( \max \{j, j'\} + 1 \big) L^y_k \Big) \\
  = \mathbb{V} \cap I^1_{(k,i)} \times  ( I^y_{(k, \min\{j, j'\})} \cup I^y_{(k, \min\{j, j'\}+1)} ),
\end{eqnarray}
see Figure~\ref{f:boxes_and_admissible}.
Note that for each $k$, the boxes $B_{(k,e)}$ and $B_{(k,e')}$ intersect if and only if the edges $e$ and $e'$ start or end at the same vertex.

Although the oriented model that we consider here does not feature a vertical random environment as before, we still need to restrict the location of the sets $S$ as we did before.
This is done through the notion of $k$-admissible set, which is introduced next.


\begin{definition}\label{d:admissible}
  For $k \geq 0$ we say that a vertex in $\{i L_k\} \times \mathbb{Z}$ is $k$-admissible if $i \in I^y_{(k,j)}$ and $i + j$ is even.
  A set $S\subseteq \mathbb{Z}$ is said to be $k$-admissible if all its vertices are $k$-admissible.
\end{definition}

Intuitively speaking, the notion of $k$-admissible generalizes to any scale $k$ the restriction $\{ i + j \text{ is even} \}$ imposed on the lattice $\mathbb{V}$, see also Figure~\ref{f:boxes_and_admissible}.

We now introduce the notion of $k$-fractal sets, which is identical to the previous Definition \ref{d:fractal} but replacing $k$-good with $k$-admissible.

\begin{definition}[$k$-fractal sets]
  \label{d:fractal2}
  Let $\xi$ be given.
  A singleton $S = \{ x \}$ is said $0$-fractal if it is $0$-admissible, or in other words if $x \in \mathbb{V}$.
  For $k \geq 1$, a subset $S$ of $\{i L_k\} \times \mathbb{Z}$ is a $k$-fractal if $S = S_1 \cup \dots \cup S_{2^\expfractal}$, where
\begin{itemize}
\item $S$ is $k$-admissible;
\item the family $\{S_0, \dots, S_{2^\expfractal}\}$ is $(k-1)$-ordered and
\item each $S_i$ is a $(k - 1)$-fractal.
 \end{itemize}
\end{definition}

\color{black}

\begin{figure}
  \centering
  \begin{tikzpicture}[scale=.5]
    \filldraw[fill=red,draw=red, opacity=.3](6,2) rectangle(8,6);
    \draw[->] (6.2,3)--(7.8,5);
    \filldraw[fill=red,draw=red,opacity=.3](.6,3) rectangle(.8,3.4);
    \filldraw[fill=red,draw=red,opacity=.3](8,4) rectangle(10,8);
    \draw[->] (8.2,5)--(9.8,7);
    \filldraw[fill=red,draw=red,opacity=.3](8,2) rectangle(10,6);
    \draw[->] (8.2,5)--(9.8,3);
    \foreach \i in {0,4,...,15}
    \draw[black, ultra thick, opacity=.5] (\i,0)--++(0,2) (\i,4)--++(0,2);
    \foreach \i in {2,6,...,15}
    \draw[black, ultra thick, opacity=.5] (\i,2)--++(0,2) (\i,6)--++(0,2);
    \foreach \i in {0,.4,...,2}
    \foreach \j in {0,0.4,...,7.9}
    \draw[black, thin] (\i,\j)--++(0,.2);
    \foreach \i in {.2,.6,...,2}
    \foreach \j in {.2,0.6,...,8}
    \draw[black, thin] (\i,\j)--++(0,.2);
  \end{tikzpicture}
  \caption{\label{f:boxes_and_admissible} The gray segments represent the admissible intervals. Note that admissible and non-admissible intervals appear alternately.
  In red, the boxes corresponding to edges are highlighted.}
\end{figure}
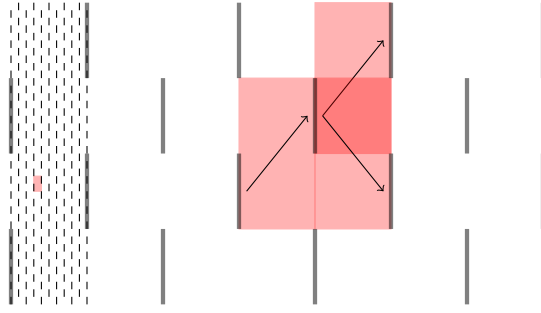

Again, observe that $k$-fractal sets have cardinality $2^{\expfractal k}$ and, in analogy to \eqref{e:grouped},
\begin{display}
  \label{e:grouped2}
  a $k$-fractal $S$ is said to be $k$-grouped if it is contained in $B_{\bar{e}}$ for some $w \in W_k$.
\end{display}
Note also that the definition of $Z_k(S)$ is kept the same.

\medskip

To define a corridor, we consider an oriented path $\gamma: \{0, \dots, n\} \to \mathbb{V}$ (by oriented, we mean simply that $(\gamma(k), \gamma(k + 1)) \in \mathbb{E}$ for every $k = 0, \dots, n - 1$) and we define the length of the above path to be $n$.

\begin{definition}
  \label{d:corridor_2}
  For a path $\gamma$ with length $n \leq \scalezero$, we define the corridor at scale $k$ over $\gamma$ as
  \begin{equation}
    C_k^\gamma := \bigcup_{i = 0}^{n - 1} B_{(k, (\gamma(i), \gamma(i + 1)))}.
  \end{equation}
\end{definition}

\paragraph{Adaptation of the proof of Theorem~\ref{t:main}}

This subsection details all the changes that need to be made for the proof of Theorem~\ref{t:main} to work for Theorem~\ref{t:oriented}.
We start by observing that both the statement and the proof of Lemma~\ref{l:corridor} work perfectly well for our oriented percolation model.
Also, the reminder $\mathcal{R}(S, R)$ can be defined just like in \eqref{e:remainder_box}.

However, we will need more care to introduce $\mathcal{R}_k(S, I)$ (in comparison with \eqref{e:remainder_box_A}) because any $k$-admissible point $x$ belongs simultaneously to two boxes at scale $k$ (corresponding to the directions $(1, 1)$ and $(1, -1)$), see Figure~\ref{f:boxes_and_admissible}.

For this adaptation, we need to first introduce the concept of parallel corridors.
Two paths $\gamma$ and $\gamma'$ are said to be parallel if:
\begin{itemize}
\item they have the same length $n$,
\item $\gamma(0)$ and $\gamma'(0)$ have the same horizontal coordinate and
\item $\gamma(k + 1) - \gamma(k) = \gamma'(k + 1) - \gamma'(k)$ for every $k = 0, \dots n - 1$,
\end{itemize}
extending this notion to corridors that are based on these paths.
See the top part of Figure~\ref{f:zig_zag} for an example of parallel corridors.

The above definition of corridors is convenient because it allows us to translate results from the non-oriented setup with very few modifications, as detailed below.

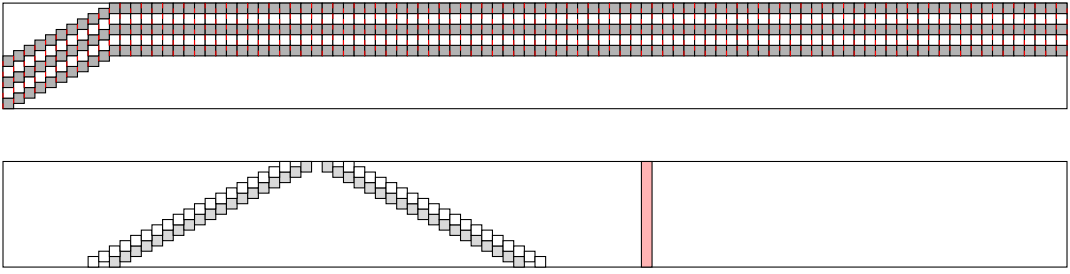
\begin{figure}
  \centering
  \begin{tikzpicture}[scale=.7]
    \draw (0,0) rectangle (20,2);
    \foreach \j in {0,2,4}{
    \foreach \i in {0,1,...,10}{
    \draw[fill=black!30!](\i*.2,\i*.1+\j*0.2) rectangle ++(.2,.2);
    \draw[red](\i*.2,\i*.1+\j*0.2) rectangle ++(0,.1);}
    \foreach \i in {11,12,...,99}
    \draw[fill=black!30!] (\i*.2,1+\j*0.2) rectangle ++(.2,.2);
    \foreach \i in {12,14,...,100}
    \draw[red](\i*.2,1+\j*0.2)-- ++(0,.1)  (\i*.2-.2,1.1+\j*0.2)-- ++(0,.1);}
     \foreach \j in {0,2}{
    \foreach \i in {0,1,...,10}
   { \draw[fill=black!0!](\i*.2,\i*.1+.2+\j*0.2)rectangle ++(.2,.2);
     \draw[red](\i*.2,\i*.1+.2+\j*0.2)rectangle ++(0,.1);}
       \foreach \i in {11,12,...,99}
       \draw[fill=black!0!] (\i*.2,1.2+\j*0.2) rectangle ++(.2,.2);
       \foreach \i in {10,12,14,...,100}
   \draw[red](\i*.2,1.2+\j*0.2)-- ++(0,.1)  (\i*.2-.2,1.3+\j*0.2)-- ++(0,.1);}
    \draw (0,-1) rectangle (20,-3);
    \foreach \i in {20,21,...,38}
    \draw[fill=black!15!](\i*.2-2,\i*.1-5)rectangle ++(.2,.2);
    \foreach \i in {20,21,...,38}
    \draw[fill=black!0!](\i*.2-.4-2,\i*.1-5)rectangle ++(.2,.2);
    \foreach \i in {0,1,...,18}
    \draw[fill=black!15!](\i*.2+6,-\i*.1-1)rectangle ++(.2,-.2);
    \foreach \i in {0,1,...,18}
    \draw[fill=black!0!](\i*.2+6.4,-\i*.1-1)rectangle ++(.2,-.2);
    \draw[fill=red!30!] (12,-1)rectangle ++(.2,-2);
  \end{tikzpicture}
  \caption{Above: The horizontal paths $\mathcal{C}_i$ of Lemma~\ref{l:branch}. The red segments correspond to the admissible sets, thus indicating in the figure the direction of the corresponding crossing for each box.
    Below: The four disjoint vertical corridors given by Lemma~\ref{l:branch}.  Note that if $I^1$ has a single defect, these four corridors can be taken all on the same side of the rectangle.}
  \label{f:zig_zag}
\end{figure}

Given a collection $\Sigma = \{\mathcal{C}_1, \dots \mathcal{C}_J\}$ of parallel $k$-corridors, observe that they all project horizontally to the same interval $I_\Sigma$.
Then, for any interval $I = [a, b] \subseteq I_\Sigma$ and a set $S \subseteq \{a\} \times \mathbb{Z}$, we introduce
\begin{equation}
  \label{e:remainder_oriented}
  \mathcal{R}^\Sigma_k(S, I) = \bigcup_{\mathcal{C}_i \in \Sigma} \mathcal{R} \Big( S \cap \mathcal{C}_i, \mathcal{C}_i \cap \big( I \times \mathbb{Z} \big) \Big).
\end{equation}
Observe that since we assumed the corridors to be parallel (and therefore disjoint) each event appearing in the above union is independent as in the previous model.

For Lemma~\ref{l:branch}, which is just a geometrical result, the restriction $i_1 - i_0 \geq L/2$ should be replaced by $i_1 - i_0 \geq L^1/2$ and the number of horizontal corridors is replaced by $L^y$, as expected.
Otherwise the statement holds in the oriented setup with minimal modifications to the corridors:
Figure~\ref{f:zig_zag} can help the reader understand the construction.

It is also important to observe that Remark~\ref{r:parallel} is valid in this context as well.
This is important to guarantee that we have parallel corridors and therefore we can use the definition $\mathcal{R}^\Sigma_k$ in \eqref{e:remainder_oriented}.

Having introduced the remainder $\mathcal{R}^\Sigma_k$, the rest of the proof of Theorem~\ref{t:oriented} follows exactly the same steps as the proof of Theorem~\ref{t:main}.
However, every appearance of $\mathcal{R}_k$ should be replaced with $\mathcal{R}^\Sigma_k$ for an appropriate choice of $\Sigma$ that we now detail in a case-by-case basis:

We should look for a pattern in the choices of $\Sigma$.
Perhaps something like this:
\begin{itemize}
\item If the interval $I$ appearing in $\mathcal{R}_k(S, I)$ $I$ is composed of a single $k$ interval, then $\Sigma$ will be composed of corridors of length one, all parallel to each other, so they all go up or they will all go down together.
  In this case the statements about $\mathcal{R}_k(S, I)$ should remain valid for either of these choices.
  Observe also that the symmetry with respect to vertical reflection implies that the (up or down) choice for the direction of these corridors does not alter the corresponding probabilities.
  This is the case for instance in the definition of $u_k$ and $v_k$, which include a maximum between up and down directions.
\item Similarly, in some cases the interval $I$ will be a $k$-block contained inside a single $k + 1$ interval.
  In this case, $\mathcal{R}_{k + 1}(S, I)$ will also have to be associated with a collection of corridors of length one and they should all go up or down.
  This is the case for instance in Definition~\ref{d:recovery} and Lemma~\ref{l:recovery}, where the bound on the probability holds for both up and down corridors.
\item In certain situations, the set $I$ in the remainder $\mathcal{R}_k(S, I)$ is a $k$-block with many $k$-intervals.
  In such cases, all the probabilities of events involving $\mathcal{R}_k(S, I)$ should be replaced by the suppremum over all collections of parallel corridors $\Sigma$.
  This is what happens, for instance, in the second claim of Lemma~\ref{l:corridor} and Lemma~\ref{l:algebra}.
\item The proof of Lemma~\ref{l:recovery} goes almost unchanged.
  In the first step (``regrouping''), whenever we mention $\mathcal{R}_k$, we use $\Sigma$ that goes up if we are trying to ``recover up'' and $\Sigma$ that goes down if we are trying to ``recover down''.
  In the second step (``branching''), we use the $\Sigma$ given by Lemma~\ref{l:branch}.
\end{itemize}

\section{Open problems}

One important class of percolation models that include two sources of randomness is given by Cox-Processes, see for instance Chapter 3 of \cite{dhillon2020poisson}.
In a Cox-Process, a random measure $\mu$ is constructed on $\mathbb{R}^2$, which is then used as the intensity of a Poisson Point Process of balls with unit radius.

If the random measure $\mu$ has good mixing properties, the phase transition for the corresponding percolation has been established by \cite{hirsch2019continuum}, see Theorem~2.4.

On the other extreme, if the random measure $\mu$ is given by a Poisson Process of lines parallel to the coordinate axes, the model becomes similar to the process studied here and has been treated in \cite{jahnel2022continuum}.
However, to the best of our knowledge it is currently open whether there is a non-trivial phase transition for the vacant set left by the Cox-Process studied in \cite{jahnel2022continuum} as we allow the balls to have random radii with non-compact support (say with an exponential distribution).

Another interesting model to consider along these lines is the Poisson Line Cox Process, see \cite{dhillon2020poisson}, for which we leave the following question.
Does it exhibit a non-trivial phase transition?
The question can be asked both for the occupied and vacant set left by the balls.

We hope that the techniques introduced in this article can help further investigate these and other related problems.

\bibliographystyle{bib/jcamsalpha}
\bibliography{bib/all}

\end{document}